\documentclass[11pt, a4paper]{article}
\usepackage[utf8]{inputenc}
\usepackage{mathtools}
\usepackage{amsmath}
\usepackage{amsthm}
\usepackage{algorithm}
\usepackage{algorithmic}
\usepackage{amssymb}
\usepackage{parskip}
\usepackage{stackengine}
\mathtoolsset{showonlyrefs}
\usepackage{tikz}
\usepackage{xcolor}
\usepackage[a4paper, margin=2.5cm]{geometry}
\usepackage{comment}
\usepackage{mathrsfs} 
\usepackage{dsfont}
\usepackage{esint}
\usepackage[export]{adjustbox}
\usepackage{multirow}
\usepackage{mathtools}
\usepackage{enumitem}
\usepackage{scalerel}
\usepackage{color}
\definecolor{hanblue}{rgb}{0.27, 0.42, 0.81}
\definecolor{mordantred19}{rgb}{0.68, 0.05, 0.0}
\definecolor{red}{rgb}{0.68, 0.05, 0.0}
\definecolor{green}{rgb}{0.0, 0.5, 0.0}
\usepackage[colorlinks, citecolor=hanblue,linkcolor=green]{hyperref}

\newcommand{\dual}{\mathbb{R}^{n \times d}_K}

\newcommand{\Kc}{K^\circ}
\newcommand{\primal}{\mathbb{R}^{n \times d}_{\Kc}}
\newcommand{\Rnd}{\R^{n\times d}}

\newcommand{\e}{\varepsilon}

\DeclareMathOperator{\supp}{supp}

\newcommand{\R}{\mathbb{R}}

\newcommand{\N}{\mathbb{N}}
\renewcommand{\S}{\mathbb{S}}
\newcommand{\ls}{\leqslant}

\newcommand{\1}{\mathds 1}
\renewcommand{\L}{\mathcal{L}}
\newcommand{\M}{\mathcal{M}}
\newcommand{\E}{\mathcal{E}}
\newcommand{\Ha}{\mathcal{H}}
\newcommand{\pa}{\partial^\ast}
\newcommand{\pe}{\partial^e}
\newcommand{\du}{{(1)}}
\newcommand{\dz}{{(0)}}

\newcommand{\TV}{\operatorname{TV}}
\newcommand{\TD}{\operatorname{TD}}
\newcommand{\BV}{\operatorname{BV}}
\newcommand{\BD}{\operatorname{BD}}
\newcommand{\TGV}{\operatorname{TGV}}
\newcommand{\SO}{\operatorname{SO}}

\newcommand{\Id}{\operatorname{Id}}

\newcommand{\tr}{\operatorname{tr}}
\newcommand{\Per}{\operatorname{Per}}
\newcommand{\Ext}{\operatorname{Ext}}

\newcommand{\Span}{\operatorname{Span}}
\renewcommand{\div}{\operatorname{div}}
\newcommand{\dd}{\, \mathrm{d}}

\newcommand{\mres}{\mathbin{\vrule height 1.4ex depth 0pt width
0.13ex\vrule height 0.13ex depth 0pt width 1.0ex}}

\newcommand\restr[2]{{
  \left.\kern-\nulldelimiterspace
  #1 
  \vphantom{\big|} 
  \right|_{#2}
  }}

\allowdisplaybreaks

\theoremstyle{plain}
\newtheorem{thm}{Theorem}
\numberwithin{thm}{section}

\newtheorem{lemma}[thm]{Lemma}
\newtheorem{prop}[thm]{Proposition}

\newtheorem{cor}[thm]{Corollary}

\theoremstyle{definition}

\newtheorem{rem}[thm]{Remark}
\newtheorem{example}[thm]{Example}
\theoremstyle{remark}

\title{On extremal points for some vectorial total variation seminorms\footnotetext{2020 Mathematics Subject Classification: 46N10, 46A55, 26B30, 49Q20.}}\vspace{-0.3cm}
\author{Kristian Bredies\thanks{Department of Mathematics and Scientific Computing, University of Graz, 8010 Graz, Austria\newline (\texttt{kristian.bredies@uni-graz.at})},\ \ 
Jos\'e A. Iglesias\thanks{Department of Applied Mathematics, University of Twente, 7500AE Enschede, The Netherlands \newline (\texttt{jose.iglesias{@}utwente.nl})},\ \ Daniel Walter\thanks{Institut f\"ur Mathematik, Humboldt-Universit\"at zu Berlin, 10117 Berlin, Germany \newline (\texttt{daniel.walter@hu-berlin.de})}}
\date{}

\begin{document}

\maketitle

\vspace{-0.4cm}

\begin{abstract}
We consider the set of extremal points of the generalized unit ball induced by gradient total variation seminorms for vector-valued functions on bounded Euclidean domains. These are central to the understanding of sparse solutions and sparse optimization algorithms for variational problems posed among such functions. For cases in which either the domain or the target are one dimensional or the sum of the total variations of each component is used, we prove that these extremals consist of piecewise constant functions with two regions. For definitions involving more involved matrix norms and in particular spectral norms, we produce families of examples to show that the resulting set of extremal points is larger and includes piecewise constant functions with more than two regions. We also consider the total deformation induced by the symmetrized gradient, for which minimization with linear constraints appears in problems of determination of limit loads in a number of continuum mechanical models involving plasticity. For this case, we show piecewise infinitesimally rigid functions with two pieces to be extremal under mild assumptions. Finally, as an example which is not piecewise constant, we prove that unit radial vector fields are extremal for the Frobenius total variation in the plane.
\end{abstract}
\vskip .3truecm \noindent Keywords:
	extremal points, vector measures, total variation, bounded deformation

\section{Introduction}

Convex and positively one-homogeneous functionals defined on appropriate Banach spaces are widely used as regularizers in variational approaches to signal and image processing, inverse problems, optimal control as well as, more recently, in continuous formulations of some machine learning methods. This is attributed to the empirical observation that their incorporation into minimization tasks leads to solutions which are given by a linear combination of a few number of ``simple'' atoms. We refer, e.g., to the tendency of LASSO-regression to produce solutions with few nonzero entries \cite{CheDonSau98} and its infinite-dimensional counterpart, the Beurling-LASSO, which is known to produce linear combinations of few Dirac masses under non-degeneracy conditions on the measurements \cite{DuvPey15}, a result that has been extended to general regularizers in \cite{CarDel23}. 

A recently popularized approach towards understanding these ``sparsifying'' properties for a generic proper, convex and positively one-homogeneous regularizer~$\mathcal{G}$ on a Banach space~$\mathcal{U}$ is its interpretation as the Minkowski, or gauge, functional of its generalized unit ball, i.e.,
\begin{align*}
    \mathcal{G}(u) = \inf \left\{\,\lambda\;\middle\vert\; \lambda \geq 0, \quad u \in \lambda \mathcal{B} \,\right\} \quad \text{where} \quad \mathcal{B} \coloneqq \left\{\,u \in \mathcal{U}\;\middle\vert\;\mathcal{G}(u) \leq 1\,\right\}.
\end{align*}
Assuming compactness of the latter in a suitable topology, the Krein-Milman Theorem allows to identify~$\mathcal{B}$  as the closed convex hull of its extremal points
\[\Ext(\mathcal{B}) = \left\{ u \in \mathcal{B} \, \middle \vert \,  u = \lambda u_1 + (1-\lambda) u_2 \text{ for }u_1,u_2 \in \mathcal{B} \text{ and } \lambda \in [0,1] \!\implies\!\! \big( u=u_1 \text{ or } u=u_2 \big) \right\}.\] 
This geometric perspective has several applications for both, the theoretical aspects of minimization problems with convex, positively one-homogeneous regularization well as their numerical realization. For example, \emph{Convex Representer Theorems},~\cite{BoyElAt19, BreCar20}, provide sufficient conditions for the existence of minimizers constituted by finitely many extremal points associated to the regularizer while problem-tailored  \emph{Generalized Conditional Gradient methods},~\cite{BreCarFanWal24}, numerically compute minimizers by relying on iterates formed as finite conic combinations of such extremals as well as greedy updates based on the minimization of linear functionals over~$\Ext(\mathcal{B})$.

For particular choices of the regularizer~$\mathcal{G}$ and the variable space~$\mathcal{U}$, the potential impact and benefit of these results heavily rely on a precise characterization of the set of extremal points. Moreover, the size and complexity of this set is also of crucial importance: A functional giving rise to a set of extremal points which is very large and/or hard to navigate would render representer theorems less informative and extremal point-based optimization methods less tractable.

\subsection{Contribution and related work}
Motivated by these potential benefits, the study of the extremal points associated to convex regularizers has become a fruitful area of research over the last years. Without pretence of completeness, we point out for a list of convex representer theorems in particular settings as well as, e.g., for realizations of extremal point-based solution algorithms.  
However, for other cases of practical relevance, a full characterization of~$\Ext(\mathcal{B})$ is not available. The present work puts the focus on total variation energies for vector-valued functions~$u \colon \Omega \to \R^n $,~$n \geq 1$, on a bounded domain~$\Omega \subset \R^d $,~$d \geq 1$, with Lipschitz boundary. More in detail, given a matrix norm~$|\cdot|_K$ (where $K$ denotes the respective closed unit ball) as well as~$u \in \BV(\Omega;\R^n)$, the space of vector-valued functions of bounded variation \cite[Def.~3.1, Prop.~3.6]{AmbFusPal00}, consider  
\begin{equation}\label{eq:tvdef} \TV_{\!K}(u) := \sup \left\{ \int_\Omega u \cdot \div \Phi~\mathrm{d}x \,\middle\vert\,  \Phi \in C^1_c(\Omega;\Rnd), \, |\Phi(x)|_{\Kc} \leq 1 \text{ for all } x\in \Omega \right\},\end{equation}
where $|\cdot|_{\Kc}$ denotes the canonical dual norm for~$|\cdot|_K$, with closed unit ball $\Kc$. Integrating by parts and using the density of $C^1_c(\Omega;\Rnd)$ in $C^1_0(\Omega;\Rnd)$, we note that this definition ensures
\begin{equation}\label{eq:TVKSobolev}\TV_{\!K}(u)=\int_\Omega |\nabla u|_K  \dd x \quad \text{for all} \quad u \in W^{1,1}(\Omega; \R^n),\end{equation} which justifies the employed notation in terms of~$K$ and further underlines the practical interest in this type of regularizer.

This general formulation captures, both, isotropic, i.e. $|\cdot|_K$ is the Euclidean norm $|\cdot|$, and anisotropic flavors, e.g.~$|\cdot|_K=|\cdot|_1$, of total variation for scalar-valued functions as well as vector-valued analogues involving e.g. Schatten $p$-norms. While the former are a fundamental pillar in the analysis of grey-scale images, see e.g. \cite{ChaCasCreNovPoc10}, the latter have been proposed, see e.g.~\cite{GolStrCre12}, as advantageous for regularizing color images and provide a suitable setting for the approximation of vector-valued functions of bounded variation by piecewise constant ones, see \cite{BabIur23}.

For seminorms such as~\eqref{eq:tvdef}, let us recall that to retain some notion of compactness of the generalized unit ball, and thus to ensure the existence of extremal points, the natural setting is to work with the quotient space w.r.t. the subspace of functions on which they vanish. Assuming $\Omega$ is connected, this means that we should work up to addition of constant functions, i.e. we consider
\begin{equation}\label{eq:defBK}\begin{aligned}
    \mathcal{B}_K \coloneqq \left\{\, \lbrack u \rbrack \in \BV(\Omega; \R^n)/(\R^n \1_\Omega)\;\middle\vert\;\TV_K(u) \leq 1\,\right\} \quad \text{where} \quad \lbrack u \rbrack=\left\{\, u+a \1_\Omega\;\middle\vert\;a\in\R^n\,\right\}.
\end{aligned}\end{equation}
For example, in~\cite[Theorem 4.7]{BreCar20}, the authors show that the extremal points of the isotropic total variation for scalar-valued functions, i.e.~$n=1$ and~$|\cdot|$, are given, up to a constant shift, by scaled characteristic functions of non-trivial sets~$E$ which are simple, i.e.~$E$ as well as~$\Omega \setminus E$ are indecomposable. While this structure is, mutatis mutandis, retained for some (vector-valued) instances of~$\TV_K$, the main aim of the present work is to show that, in general,~$\Ext(\mathcal{B}_K)$ is larger than one would extrapolate from the scalar-valued case and its complexity strongly depends on~$|\cdot|_K$. More in detail, we provide the following results:
\begin{itemize}
    \item We give a full characterization of the set of extremal points associated to~$\TV_K$ for the case of vector-valued functions on an interval~$\Omega=(a,b)$,
    \begin{align*}
        \Ext \left( \mathcal{B}_K \right)=\left\{\,\left\lbrack Q\1_{(x,b)} \right\rbrack\;\middle\vert\;x \in \Omega,~Q \in \operatorname{Ext}(K)\,\right\} \quad \text{where} \quad  K \coloneqq \left\{\, Q \in \R^n \;\middle\vert\;|Q|_K\leq 1\,\right\}   
    \end{align*}
    as well as for the scalar-valued case induced by a general norm $|\cdot|_K$ on $\R^d$,
    \begin{align} \label{eq:introscalar}
        \Ext \left( \mathcal{B}_K \right)=\left\{\,\left\lbrack \frac{\1_E}{\TV_K(\1_E)}\right\rbrack\;\middle\vert\; E\subset \Omega \ \ \text{simple}, \ |E| \in \big(0,|\Omega|\big)\right\}
    \end{align}
    on multidimensional domains,~$n=1$ and~$d \geq 1$. The latter might seem particularly surprising at first glance since anisotropic total variation regularization tends to favor different structural properties in the reconstructions compared to its isotropic counterpart. From this perspective, the characterization in~\eqref{eq:introscalar} implies that this feature of anisotropic total variation cannot be concluded (solely) from convex representer theorems.
    \item The main part of the paper revolves around the general vector-valued case on multidimensional domains. On the one hand, we show that the characterization of~$\Ext(\mathcal{B}_K)$ can be traced back to the scalar-valued case if~$\TV_K$ is linear w.r.t. to a decomposition of~$\BV(\Omega)$ in a direct sum of~$n$ subspaces. For example, given an arbitrary norm~$|\cdot|_k$ on~$\R^d$, we obtain for \[|A|_K = \sum^n_{j=1} |A_{j}|_k\] that
    \begin{align*}
        \Ext(\mathcal{B}_K)= \left\{\,\left\lbrack e_j \frac{\1_E}{\TV_k(\1_E)}\right\rbrack\;\middle\vert\; E \subset \Omega~\text{simple}, \ |E| \in \big(0,|\Omega|\big),~j=1,\dots,n\, \right\},
    \end{align*}
    with~$\{e_j\}^n_{j=1}$ denoting the canonical basis of~$\R^n$ and $A_j$ the corresponding rows of $A$.
    On the other hand, we prove that~$\Ext(\mathcal{B}_K)$ is significantly larger if the matrix norm $|\cdot|_K$ is highly isotropic on rank-one matrices, see \eqref{eq:rankoneiso} for the exact condition. Denoting by~$\sigma_1(A) \geq \dots \geq \sigma_d (A) \geq 0  $ the singular values of~$A\in\Rnd$, these results cover, both, Schatten as well as Ky-Fan norms,
    \begin{align*}
        |A|_K=\bigg({\sum^d_{j=1} \sigma_j(A)^p} \bigg)^{1/p} \quad \text{and} \quad |A|_K= \sum^N_{j=1} \sigma_j(A)
    \end{align*}
    for~$p \in [1,+\infty)  $ and~$1\leq N \leq d$.
Explicit families of extremal points are constructed which are not multiples of indicatrices of simple sets. While this does not provide a full characterization of the extremal points, it emphasizes that the structure of elements in~$\Ext(\mathcal{B}_K)$ in the vector-valued setting critically depends on the particular choice of the matrix norm~$|\cdot|_K$ when acting on the matrix-valued Radon measure~$Du$. In particular, strict convexity of~$|\cdot|_K$ along certain kinds of variations tends to induce larger sets of extremals.
    \item Relating to mechanical problems involving plasticity and higher-order functionals such as the total generalized variation ($\TGV$, see \cite{BreKunPoc10}), it is also of interest to investigate the extremal points of the closed unit ball of the total deformation associated to a matrix norm~$|\cdot|_K$,
        \begin{equation}\label{eq:tddef}\TD_K(u) := \sup \left\{ \int_\Omega u \cdot \div \Phi \dd x \,\middle\vert\,  \Phi \in C^1_c(\Omega;\R^{d\times d}_{\mathrm{sym}}), \, |\Phi(x)|_{\Kc} \leq 1  \text{ for all } x\in \Omega \right\},\end{equation}
        which is finite for~$u \in \BD(\Omega;\R^d)$, the space of vector fields of bounded deformation \cite{Tem83, TemStr80}. In this case, we prove that functions of the form $w \1_E$ with $w \in \mathcal{A}$ give rise to extremal points of the ball
    \begin{align*}
    \left\{ [u] \in \BD(\Omega)/\mathcal{A} \;\middle\vert\; \TD_K(u) \leq \TD_K(u_H)\right\},
    \end{align*}
    where $\mathcal{A}$ is the set of infinitesimal rigid motions, defined in \eqref{eq:infrigid} below, and $|\cdot|_K$ is assumed to satisfy a certain strict convexity condition along symmetrized tensor products, cf. \eqref{eq:symrankoneconv}. This fact can be of interest for some continuum mechanical models involving plasticity, in which the norm used for $\TD_K$ corresponds to the yield criterion in stress space, a property of the materials considered.
    \item Finally, inspired by the extremality of~$(1-|x|)_+$ for the Hessian-Schatten variation proved in~\cite{AmbBreCon23}, we show that the planar vector field~$u_H$ defined on ~$\Omega=B_1(0) \subset \R^2$ by 
    \begin{align*}
    u_H(x):=\begin{cases}\frac{x}{|x|} & \text{ if } x \neq (0,0) \\ (0,0 ) & \text{ if } x = (0,0)\end{cases}
    \end{align*}
    satisfies
    \begin{align*}
        \left\lbrack \frac{u_H}{\TV_K(u_H)} \right \rbrack \in \Ext(\mathcal{B}_K) \quad  \text{where} \quad |\cdot|_K=|\cdot|_F
    \end{align*}
    denotes the Frobenius norm.
\end{itemize}

In the existing literature there are several examples of generalized conditional gradient algorithms explicitly using characterizations of extremal points for specific instances and generalizations of the total variation. These include \cite{TraWal23} for $\TV_K$ in one dimension, \cite{IglWal22} for total generalized variation also in one dimension, \cite{CasDuvPet23} for isotropic total variation in two dimensions, and \cite{CriIglWal23} for piecewise constant discretizations on two- and three-dimensional triangulations. 

Minimization of the total deformation $\TD_F$ with zero divergence in perforated domains with nonhomogeneous Dirichlet boundary conditions is considered in \cite{IglMerChaFri20} for computing the critical yield number and limit flow profile induced by solid particles settling in viscoplastic fluids. For the corresponding scalar-valued $\TV$ problem arising from anti-plane motions, in \cite{FriIglMerPoeSch17, IglMerSch20} the existence of piecewise constant solutions is proved. Particularly in \cite{IglMerSch20}, which treats the case of multiple solid components, the extremality of rescaled indicatrices of simple sets is also directly used. In the axisymmetric case with in-plane motions leading to $\Omega \subset \R^2$ and specific particle geometries, vector fields with restrictions of the form $x=(x_1, x_2) \mapsto (-x_2, x_1)/|x|$ can be numerically observed when the solid particle has the shape of a square with a corner pointing to the direction of gravity. We are not able to prove extremality in this case, but discuss some partial results in Remark \ref{rem:noBDhedgehog}.

\subsection{Outline} \label{subsec:outline}
After summarizing some preliminaries as well as the necessary notation in Section~\ref{sec:notation}, we start by characterizing the extremal points for the~$\TV_K$-functional in the particular cases of vector-fields on an interval as well as scalar-valued functions of total variation on domains in multiple dimensions, see Section~\ref{section:onescalar}. Subsequently, in Section~\ref{sec:additive}, we show that the structure of the latter is retained in the vector-valued setting if the energy~$\TV_K$ is additive w.r.t to some decomposition of~$\BV(\Omega;\R^n)$ into a direct sum of subspaces. In Section~\ref{sec:gennorm}, we observe that the set of extremal points associated to~$\TV_K$ and a rotationally invariant~$|\cdot|_K$, e.g. any spectral norm, is significantly larger than what we would expect from the scalar-valued case, and that in particular, the set~$\mathcal{B}_K$ contains families of extremal points attaining more than two values. Moreover, still in Section~\ref{sec:gennorm}, we also prove extremality with respect to $\TD_K$ of most functions made by gluing two infinitesimal rigid motion along the reduced boundary of a simple set. Finally, in Section \ref{sec:hedgehog}, we prove that the unit radial vector field on the unit ball~$\Omega=B(0,1) \subset \R^2$ (we denote by $B(x,r)$ the Euclidean ball of radius $r >0$ centered at $x \in \R^d$) corresponds to an extremal point for $\TV_K$, if~$|\cdot|_K$ is given by the Frobenius norm.
\section{Notation} \label{sec:notation}
\paragraph*{Matrix norms.}
The central notion in this article is the total variation for locally integrable functions $u:\Omega \to \R^n$ on $\Omega \subset \R^d$, defined as in \eqref{eq:tvdef}. In it, we have used a dual pair of matrix norms $|\cdot|_{K}$ and $|\cdot|_{\Kc}$, meaning that for all for $A \in \Rnd$ we have
\begin{equation}\label{eq:Kduality}\begin{aligned}|A|_{\Kc} &= \sup \left\{ \tr\big(A^\top B\big)\,\middle\vert\, B \in \Rnd, \ |B|_K \ls 1\right\},\text{ and} \\ |A|_K &= \sup \left\{ \tr\big(A^\top B\big)\,\middle\vert\, B \in \Rnd, \ |B|_{\Kc} \ls 1\right\}.
\end{aligned}\end{equation}
In this context, $K$ always denotes the closed unit ball with respect to the norm $|\cdot|_K$, that is the symmetric convex set
\[K = \left\{ A \in \Rnd \,\middle\vert\, |A|_K \leq 1 \right\},\]
which indeed makes $|\cdot|_K$ the Minkowski or gauge functional of $K$, that is 
\[|A|_K = \inf \{ r \geq 0 \,\vert\, M \in rK \}.\]

The set $\Kc$ is then the polar set of $K \subset \Rnd$ with respect to the Frobenius inner product, that is
\[\Kc := \left\{A \in \Rnd \,\middle\vert\, \tr\big(A^\top B\big) \leq 1 \text{ for all }B \in K\right\}.\]
This notation is chosen to be consistent with it being the unit ball for the dual norm \eqref{eq:Kduality}, that is
\[\Kc = \left\{ B \in \Rnd \,\middle\vert\, |B|_{\Kc} \leq 1 \right\}.\]
An easy to interpret family of matrix norms which we use in the sequel is that given by
\begin{equation}\label{eq:KsKvNorm}|A|_K = \Big| \big(|A_1|_{K_s},\ldots,|A_n|_{K_s}\big)^\top\Big|_{K_v},\end{equation}
defined in terms of a ``space'' ball $K_s \subset \R^d$ acting on the rows $A_i$ of $A$ (that is, the distributional gradients of each component of $u$ when using $\TV_K$) and a ``value/vector ball'' $K_v \subset \R^n$. In this situation, one can write the dual norm to $|\cdot|_K$ above as 
\[|A|_{K^\circ} = \Big| \big(|A_1|_{K^\circ_s},\ldots,|A_n|_{K^\circ_s}\big)^\top\Big|_{K^\circ_v}.\]
To see this, notice that in the duality formula
\begin{equation}\label{eq:polarset}|A|_{\Kc} = \sup \left\{ \tr\big(A^\top B\big)\,\middle\vert\, B \in \Rnd, \ |B|_K \ls 1\right\}, \text{ we have }\tr\big(A^\top B\big)=\sum_{i=1}^n A_i \cdot B_i,\end{equation}
which allows us to first find a vector in $\R^n$ to realize the outer duality between $|\cdot|_{K_v}$ and $|\cdot|_{K^\circ_v}$, and extend to $\R^d$ the functionals that assign the value of the components of this vector to each of the rows of $A$, realizing the inner duality between $|\cdot|_{K_s}$ and $|\cdot|_{K^\circ_s}$ for each row. In an analogous fashion, one can also consider instead
\begin{equation}\label{eq:KvKsNorm}|A|_K = \Big| \big(\big|A^1\big|_{K_v},\ldots,\big|A^d\big|_{K_v}\big)^\top\Big|_{K_s},\end{equation}
where instead the ``values'' ball $K_v \subset \R^n$ acts first on the columns $A^i$ of $A$. The dual norm becomes then
\[|A|_{K^\circ} = \Big| \big(|A^1|_{K^\circ_v},\ldots,|A^d|_{K^\circ_v}\big)^\top\Big|_{K^\circ_s},\]
as can be readily seen by using $\tr(A^\top B) = \tr(A B^\top)$ in \eqref{eq:polarset}.

For simplicity we single out two particular cases, denoting $|\cdot|$ for the Euclidean norm of vectors and $|\cdot|_F$ for the Frobenius norm of matrices, defined by $|A|^2_F = \tr\big(A^\top A\big)$ for any $A \in \R^{n \times d}$.

Often we will make use of rank-one matrices, which can be written as tensor products $b \otimes a := b a^\top \in \R^{n \times d}$ generated by $a \in \R^d$ and $b \in \R^n$, as well as their symmetrized counterparts given by $a \odot b := \frac12 ( a \otimes b + b \otimes a )$ for $a,b \in \R^d$. We further note that for rank-one matrices, norms of the form \eqref{eq:KsKvNorm} can be computed as
\[|b \otimes a|_K = \Big| \big(b_1 |a|_{K_s},\ldots,b_n|a|_{K_s}\big)^\top\Big|_{K_v} = |b|_{K_v} |a|_{K_s},\]
which in turn equals the value obtained for norms of the type \eqref{eq:KvKsNorm}.

\paragraph*{Quotient spaces.}
Let us recall that to retain compactness of the unit ball for seminorms such as~$\TV_K$ and $\TD_K$, one needs to work with the quotient space with respect to the subspace of functions on which they vanish. Assuming that $\Omega$ is connected, this means that for $\TV_K$ we should work up to addition of constant functions, in $\BV(\Omega; \R^n)/(\R^n \1_\Omega)$. For $\TD_K$ the kernel is larger and consists of all infinitesimal rigid motions, meaning that we should work in $\BD(\Omega)/\mathcal{A}$, for 
\begin{equation}\label{eq:infrigid}\mathcal{A} := \left\{ v:\Omega \to \R^d \,\middle\vert\, v(x) = Ax + b \ \text{ for } A \in \R^{d\times d} \ \text{ with }A+A^\top = 0,\,\text{ and } b \in \R^d\right\}.\end{equation}
In both cases, we use the notation $[\cdot]$ to denote equivalence classes in a quotient space.

\paragraph*{Vector-valued measures.}
It will also be useful to separate the effect of the derivative and consider anisotropic norms on vector-valued measures. In the following we identify the dual of the finite dimensional space~$\primal=(\Rnd, {|\cdot|_{\Kc}})$ with $\dual = (\Rnd, |\cdot|_{K})$ via the Frobenius inner product $A \cdot B=\tr(A^\top B)$. Now, let~$C_0(\Omega; \primal)$ denote the set of~$\primal$-valued continuous functions on~$\Omega$ which vanish at its boundary. Together with the canonical maximum norm, i.e.,
\begin{align*}
    \|\psi\|_{C,\primal}=\max_{x \in \Omega} |\psi(x)|_{\Kc} \quad \text{for all} \quad \psi \in C_0(\Omega; \Rnd)
\end{align*}
these form a separable Banach space. By Singer's representation theorem \cite[pp.~182]{DieUhl77}, we identify its dual space with~$M(\Omega;\dual)$, the space of~$\dual$-valued vector measures with finite total variation. The corresponding duality pairing is given by
\begin{align*}
    \langle \psi, \mu \rangle= \int_\Omega \psi \cdot \mathrm{d} \mu= \int_\Omega \tr\left(\psi(x)^\top \frac{\mu}{|\mu|_K}(x)\right)~\mathrm{d}|\mu|_K(x),  
\end{align*}
for all~$\psi \in C_0(\Omega; \primal)$,~$\mu \in M(\Omega; \dual)$,
where~$|\mu|_K$ denotes the total variation measure of $\mu$ with respect to $|\cdot|_K$, see, e.g., \cite[Def.~1.4]{AmbFusPal00}), and the notation $\mu/|\mu|_K$ stands for the Radon-Nikod\'ym derivative of $\mu$ with respect to $|\mu|_K$, that is, $\mu/|\mu|_K$ is the polar of $\mu$. The latter satisfies~$|\mu/|\mu|_K(x)|_K=1$ for~$|\mu|_K$-a.e.~$x \in \Omega$. Equipping~$M(\Omega;\dual)$ with the canonical dual norm
\begin{align*}
    \|\mu\|_{M,\dual}= \sup_{\|\psi\|_{C,\primal}\leq 1} \int_\Omega \tr\left(\psi(x)^\top \frac{\mu}{|\mu|_K}(x)\right)~\mathrm{d}|\mu|_K(x) \quad \text{for all} \quad \mu \in M(\Omega;\dual),
\end{align*}
makes it a Banach space and we have $\|\mu\|_{M,\dual} = |\mu|_K(\Omega)$. To reflect the above mentioned special cases, we use the notation~$|\mu|$ and~$|\mu|_F$ for the total variation measure of~$\mu \in M(\Omega; \R^n)$ w.r.t the Euclidean norm~$|\cdot|$ and of~$\mu \in M(\Omega; \R^{n\times d})$ w.r.t the Frobenius norm~$|\cdot|_F$, respectively.

Finally, for a function~$u \colon \Omega \to \R^n$ of bounded variation (bounded deformation, respectively) its (symmetrized) distributional derivative is a finite Radon measure and we have~$\TV_K (u)=|Du|_K(\Omega)$ ($\TD_K (u)=|\E u|_K(\Omega)$).

\paragraph*{Subsets, relative perimeter, and piecewise functions.}
Throughout the paper, we will work extensively with Borel sets $O \subset \Omega \subset \R^d$, the collection of which we denote by $\mathfrak{B}(\Omega)$. For such a set, we denote by $|O|$ its Lebesgue measure, that is $|O| = \L^d(O)$. Further, denoting its characteristic function in $\Omega$ as $\1_E: \Omega \to \{0,1\}$, we define $\Per(E,\Omega)$ to be its isotropic perimeter in $\Omega$, given by
\[\Per(E,\Omega) := \TV(\1_E),\]
where $\TV$ denotes the particular instance of \eqref{eq:tvdef} in which $n=1$ and $K$ is the closed unit ball of the Euclidean norm $|\cdot|$ in $\R^d$. Whenever $\Per(E, \Omega) < +\infty$ we say that $E$ is a set of finite perimeter $E$ in $\Omega$, in which case $\Per(E,\Omega) = |D\1_E|(\Omega)$.

There are several notions of boundary which are relevant for sets of finite perimeter. The topological boundary is not a priori useful, since the definition of perimeter is invariant to modifications of Lebesgue measure zero. On the other hand, for the support $\supp D\1_E$ of the Radon measure $D\1_E$, defined to be the intersection of all closed sets $F$ such that $|D\1_E|(F)=|D\1_E|(\Omega)$, we have \cite[Prop.~12.19]{Mag12} that 
\[\supp D\1_E = \left\{ x \in \Omega \,\middle\vert\, 0 < \frac{|E \cap B(x,r)|}{|B(x,r)|} < 1 \text{ for all }r>0\right\},\]
and we may take a representative $\widetilde{E}$ such that $|\widetilde{E} \Delta E| = 0$ and $\partial \widetilde{E} \cap \Omega = \supp D\1_E$. Related to these considerations are the points of density $s \in [0,1]$ with respect to the Lebesgue measure, denoted as
\begin{equation}E^{(s)}:= \left\{x \in \R^d \,\middle\vert\, \lim_{r \to 0^+} \frac{|E \cap B(x,r)|}{|B(x,r)|}=s\right\},\end{equation}
including the special cases of the measure-theoretic exterior $E^{(0)}$ and measure-theoretic interior $E^{(1)}$. Finally, the reduced boundary $\pa E \subset \Omega$ is defined to be the set of points $x \in \supp D\1_E$ such that
\begin{equation}\lim_{r \to 0}\frac{D\1_E(B(x,r))}{|D\1_E|(B(x,r))} \text{ exists and belongs to }\S^{d-1},\end{equation}
where the limit is referred to as the measure-theoretic (inner) normal to $E$ at $x \in \pa E$, and is denoted as $\nu_E(x) \in \S^{d-1}$. By the de Giorgi structure theorem \cite[Thm.~15.9]{Mag12} and the Federer theorem \cite[Thm.~16.2]{Mag12} we have that
\begin{equation}\label{eq:D1E-Ha}D\1_E = \nu_E \Ha^{d-1}\mres \pa E, \quad\text{and }\ \Ha^{d-1}\big( E^{(1/2)} \setminus \pa E \big) = 0,\end{equation}
implying in particular that $\Per(E,\Omega) := \Ha^{d-1}(\pa E)$. 

We remark that even though all our results are stated for bounded domains $\Omega$, we have chosen to keep the second argument of the perimeter since, unlike for functions $u:\Omega \to \R^n$ and $\TV_K(u)$, the domain $\Omega$ cannot be inferred from specifying only the subset $E$, and the notation $\Per(E)$ is commonly used in the literature to mean $\Per(E,\R^d)$. The above definition of perimeter in $E$ is extended to the relative perimeter of $E$ in a Borel set $F \subset \Omega$ as
\[\Per(E,F) := |D\1_E|(F) = \Ha^{d-1}(\pa E \cap F).\]
In the sequel, quantities of the form $\TV_K(\1_E) = |D\1_E|_K(\Omega)$ for other norms $|\cdot|_K$ on $\R^d$ will also make an appearance, but we reserve the perimeter notation for the standard isotropic version above.

For a finite perimeter set $E$ in $\Omega$ we further say that it is indecomposable if, whenever
\[E = E_1 \cup E_2 \ \text{ with }|E_1 \cap E_2| = 0\ \text{ and }\Per(E,\Omega) = \Per(E_1,\Omega) + \Per(E_2,\Omega),\]
then necessarily $|E_1|=0$ or $|E_2|=0$. Further, $E$ is called simple if both $E$ and $\Omega \setminus E$ are indecomposable. 

For a function $u \in \BV(\Omega; \R^n)$, we say that it is piecewise constant \cite[Def.~4.1]{AmbFusPal00} if there is a partition of $\Omega$ into subsets $\{E_i\}_{i =1}^{\infty}$, and vectors $b_i \in \R^n$ for which
\[\sum_{i=1}^{\infty}\Per(E, \Omega) < + \infty \quad\text{and}\quad u = \sum_{i=1}^{\infty} b_i \1_{E_i}.\]
Similarly, we say that $u \in \BD(\Omega)$ is piecewise infinitesimally rigid if it can be written as 
\[u = \sum_{i=1}^{\infty} v_i \1_{E_i} \quad\text{for }v_i \in \mathcal{A}\]
with the same (Caccioppoli) type of partition $\{E_i\}_{i=1}^{\infty}$. Moreover, we remark that we will often write equalities of functions defined on $\Omega$ to mean their equivalence, that is, equality $\L^d-$almost everywhere.

\section{The one-dimensional and scalar-valued cases} \label{section:onescalar}
We start by giving a precise characterization of the extremal points of the unit ball
\[ \mathcal{B}_K=\left\{ \lbrack u \rbrack \in \BV(\Omega; \R^n)/(\R^n \1_\Omega) \,\middle\vert\, \TV_K(u) \leq 1 \right\}\]
associated to~$\TV_K$, for the particular cases of vector-valued functions on an interval, i.e.~$d=1$ and~$n \geq 1$, as well as scalar-valued functions on domains in multiple dimensions,~$d \geq 1$ and~$n=1$.  

\subsection{The one-dimensional case} \label{subsec:onedim}
Loosely speaking, if~$\Omega=(0,T) \subset \R$ is an interval, we can rely on taking primitives of vector measures~$\mu \in M(\Omega;\dual)$ in order to obtain~$\Ext(\mathcal{B}_K)$.   
More in detail, consider the linear mapping
\begin{align} \label{defL}
    L \colon \BV(\Omega; \R^n)/(\R^n \1_\Omega)  \to M(\Omega;\dual), \quad [u] \mapsto Du,
\end{align}
where $D$ denotes the distributional derivative and the quotient space~$\BV(\Omega; \R^n)/(\R^n \1_\Omega)$ is equipped with the canonical norm
\begin{align*}
    \|[u]\|_{\BV / \R^n}= \inf_{c \in \R^n} \left \lbrack \|u-c\1_\Omega\|_{L^1}+ \TV(u) \right \rbrack= \inf_{c \in \R^n} \left \lbrack \|u-c\1_\Omega\|_{L^1}+ \|Du\|_{M,\dual} \right \rbrack.
\end{align*}
\begin{lemma}
Assume that~$d=1$. The linear mapping~$L$ from~\eqref{defL} is continuous and bijective.
\end{lemma}
\begin{proof}
Given~$\mu \in M(\Omega;\dual)$, we readily see that the function~$u(t)=\mu((0,t))$, defined in an a.e. sense, satisfies~$Du=\mu$. Hence,~$L$ is surjective. In order to show that~$L$ is also injective, let~$[u_1],[u_2] \in\BV(\Omega; \R^n)/(\R^n \1_\Omega)$ satisfy~$L([u_1])=L([u_2])$. Then we have~$Du_1=Du_2$ and thus~$u_1=u_2+b$, for some~$b\in\R^n$, hence $[u_1] = [u_2]$ by definition of the equivalence class. The continuity of~$L$ follows immediately noting that
\[\|L([u])\|_{M,\dual}=\|Du\|_{M,\dual} \leq \|[u]\|_{\BV / \R^n} \quad \text{for all} \quad [u] \in \BV(\Omega; \R^n)/(\R^n \1_\Omega).\qedhere\]
\end{proof}
In particular, we observe that
\begin{align*}
    L^{-1} \left( \left\{ \mu \in M(\Omega;\dual) \,\middle\vert\, \|\mu\|_{M,\dual} \leq 1 \right\}\right)= \left\{ \lbrack u \rbrack \in \BV(\Omega; \R^n)/(\R^n \1_\Omega) \,\middle\vert\, \TV_K(u) \leq 1 \right\}
\end{align*}
and thus
\begin{align*}
    & L^{-1} \left( \Ext \left( \left\{ \mu \in M(\Omega;\dual) \,\middle\vert\, \|\mu\|_{M,\dual} \leq 1 \right\}\right)\right) \\ & \quad = \Ext \big( \left\{ \lbrack u \rbrack \in \BV(\Omega; \R^n)/(\R^n \1_\Omega) \,\middle\vert\, \TV_K(u) \leq 1 \right\} \big)
\end{align*}
according to~\cite[Lemma~3.2]{BreCar20}. As a consequence, it suffices to characterize the extremal points of the unit ball in~$M(\Omega;\dual)$.
\begin{lemma} \label{lem:charvecmeas}
There holds
\begin{align*}
    \Ext \left( \left\{ \mu \in M(\Omega;\dual) \,\middle\vert\, \|\mu\|_{M,\dual} \leq 1 \right\}\right)= \left\{\,b \delta_t\,\middle\vert\,t \in \Omega,~b \in \Ext(K)\,\right\}.
\end{align*}
\end{lemma} 
\begin{proof}
We start by noting that both,~$|\cdot|_K$ and~$\|\cdot\|_{M,\dual}$, are norms. Hence, vectors~$b \in \Ext(K)$ and measures 
\begin{align*}
  \mu \in \Ext \left( \left\{ \mu \in M(\Omega;\dual) \,\middle\vert\, \|\mu\|_{M,\dual} \leq 1 \right\}\right)
\end{align*}
necessarily satisfy~$|b|_K=1$ and~$\|\mu\|_{M,\dual} = 1$ respectively. 

Now, we first show that every measure of the form~$\mu=  b \delta_t$,~$b\in \Ext(K)$,~$t \in \Omega = (0,T)$, is an extremal point.
For this purpose, define the following set of scalar-valued measures:
\begin{align*}
    \mathcal{P} \coloneqq \left\{\,\nu \in M(\Omega)\,\middle\vert\,\nu \geq 0,~\nu(\Omega)\leq 1\,\right\}
\end{align*}
where~$ \nu \geq 0$ is understood in the canonical sense, i.e., by testing against nonnegative functions. 
Following the arguments of~\cite[Proposition~4.1]{BreCar20}, we see that~$\Ext ({\mathcal{P}})=\{\,\delta_t\,\vert\,t \in \Omega\,\} \cup \{0\}$.
By definition, further note that
\begin{align*}
   \|\mu\|_{M,\dual}= |\mu|_K(\Omega)=|b|_K = 1
\end{align*}
since~$b \in \Ext(K)$.

We now argue by contradiction: If~$\mu=  b \delta_t$,~$b\in \Ext(K)$,~$t \in \Omega = (0,T)$ is not extremal, then there are~$\mu_1,\mu_2 \in M(\Omega;\dual)$,~$\mu_1\neq \mu_2$, as well as~$\lambda\in(0,1)$ with
\begin{align*}
 \mu=(1-\lambda)\mu_1+\lambda \mu_2, \quad \|\mu_1\|_{M,\dual} \leq 1,~\|\mu_2\|_{M,\dual} \leq 1.
\end{align*}
Since $|\mu|_K(\Omega)=1$, we may also assume  that~$|\mu_1|_K(\Omega)=|\mu_2|_K(\Omega)=1$. This implies
\begin{align} \label{eq:probabilityhelp}
    0 \leq |\mu|_K(E)=|(1-\lambda)\mu_1+\lambda \mu_2|_K(E) \leq (1-\lambda)|\mu_1|_K(E)+\lambda |\mu_2|_K(E) \leq 1.
\end{align}
for all Borel sets~$E \subset \Omega$. Inserting~$E=\Omega$ and noting that~$|\mu|_K(\Omega)=1$, we conclude that~$\gamma \coloneqq (1-\lambda)|\mu_1|_K+\lambda |\mu_2|_K$  satisfies~$\gamma \geq 0$ and~$\gamma(\Omega)=1$. Due to~$\delta_t(E) \leq \gamma(E)$ for all Borel sets $E \subset \Omega$, we then get that $\gamma=\delta_t$.
Since~$\delta_t$ is an extremal point of~$\mathcal{P}$ and~$|\mu_1|_K,|\mu_2|_K \in \mathcal{P}$, this yields~$|\mu_1|_K=|\mu_2|_K=\delta_t$. Thus, there are~$b_1,b_2 \in K$ with
\begin{align*}
    \mu_1=  b_1 \delta_t, \quad \mu_2= b_2 \delta_t \quad \text{as well as} \quad \mu= ((1-\lambda)b_1+\lambda b_2)\delta_t.
\end{align*}
Since~$b \in \operatorname{Ext}(K)$, we finally conclude~$b_1=b_2=b$ and thus also~$\mu_1=\mu_2=\mu$, yielding a contradiction.

Next, we show that there are not more extremal points besides the mentioned Dirac delta functionals. For this purpose, first assume that~$\mu$ is extremal but its support is not a singleton. Then we have~$|\mu|_K(\Omega)=1$ and there is a Borel set~$E$ with~$0<|\mu|_K(E)<1$. Setting
\begin{align*}
    \lambda= |\mu|_K (\Omega \setminus E) \in (0,1), \quad \mu_1= \frac{1}{1-\lambda} \mu \mres E \quad \text{as well as} \quad  \mu_2= \frac{1}{\lambda} \mu \mres (\Omega \setminus E), 
\end{align*}
we note that~$|\mu_1|_K=|\mu_2|_K=1$,~$\mu_1 \neq \mu_2$ and
\begin{align*}
    \mu=  \mu \mres E+ \mu \mres (\Omega \setminus E)= (1-\lambda) \mu_1 +\lambda \mu_2 .
    \end{align*}
This contradicts the extremality of~$\mu$. Hence, every extremal point is of the form~$\mu=b \delta_t$ for some~$b\in K$. Finally, if~$b \in K \setminus \Ext(K)$, there are~$b_1,b_2 \in K$,~$b_1\neq b_2 $, and~$\lambda\in (0,1)$ with~$b=(1-\lambda)b_1+\lambda b_2$. Consequently,~$\mu$ is again not extremal.
\end{proof}

\begin{thm} \label{thm:charinone}
Let~$d=1$ and $\Omega=(0,T)$. Then there holds
\begin{align*}
    \Ext \left( \mathcal{B}_K \right)=\left\{\,\left\lbrack b\1_{(t,T)} \right\rbrack\;\middle\vert\;t \in (0,T),~b \in \operatorname{Ext}(K)\,\right\}.
\end{align*}
\end{thm}
\begin{rem}
It is worth pointing out that Lemma~\ref{lem:charvecmeas} also holds in higher-dimensional domains~$\Omega$. However, for~$d \geq 1$,~$L$ is not bijective and thus a direct transfer of the proposed strategy to domains in multiple dimensions is not possible.  
\end{rem}

\subsection{The scalar-valued case} \label{subsec:scalarcase}
Next, we consider the case of scalar-valued functions of bounded variation on domains in multiple dimensions, that is $n=1$ and $d>1$. For this purpose, we first recall the isotropic case, i.e.~$K$ is given by the Euclidean unit ball in~$\R^d$. In this context, it is well known, see e.g.~\cite{BreCar20}, that the set of extremal points is constituted by the equivalence classes of (scaled) characteristic functions of simple sets. The proof of this characterization inherently relies on the equivalence of isotropic total variation of an indicator function of a set and the perimeter of the latter. Surprisingly, the following results shows that the structure of~$\Ext(\mathcal{B}_K)$ is, mutatis mutandis, retained in the anisotropic case of~\eqref{eq:tvdef}, i.e.,
    \begin{align*}
        \Ext(\mathcal{B}_K)= \left\{\, \left\lbrack\frac{\1_E}{\TV_K(\1_E)}\right\rbrack\;\middle\vert\;E~\text{simple}, \ |E| \in \big(0,|\Omega|\big)\right\}.
    \end{align*}    
    In its proof, we crucially rely on representing the anisotropic total variation functional~$\TV_K$ in a ``primal'' form
\begin{equation}\label{eq:TVKfromnormals}\TV_K(u) = \begin{cases}\displaystyle{\int_\Omega \left|\frac{Du}{|Du|}(x)\right|_K \dd|Du|(x)} &\text{ if } u \in \BV(\Omega)\\ +\infty &\text{ otherwise},\end{cases}\end{equation}
rather than by its original definition by duality in~\eqref{eq:tvdef}. Here, $Du/|Du|$ is the polar with respect to the usual (Euclidean) definition of $|Du|$. In the scalar case and allowing for space dependence of $K$ which we do not take into account here, this equivalence is proved in \cite[Thm.~5.1]{AmaBel94} (see also Lemma \ref{lem:TVKprimal} below). Moreover, $\TV_K$ also satisfies a coarea formula, namely
\begin{equation}\label{eq:TVKcoarea}\int_\Omega \left|\frac{Du}{|Du|}(x)\right|_K \dd|Du|(x) = \int_{-\infty}^{+\infty} \int_\Omega \left|\frac{D\1_{\{u > t\}}}{|D\1_{\{u > t\}}|}(x)\right|_K \dd|D\1_{\{u > t\}}|(x) \dd t,\end{equation}
as noticed in \cite[Rem.~4.4]{AmaBel94} and proved in \cite[Prop.~2.3.7]{Jal12}. Furthermore, note that for the integrand in~\eqref{eq:TVKfromnormals} we have the identity
\begin{align*}
    |Du|_K(O)= \int_O \left|\frac{Du}{|Du|}(x)\right|_K \dd|Du|(x) \quad \text{for all} \quad O \in \mathfrak{B}(\Omega).
\end{align*}
According to \cite[Prop.~1.2]{DemTem84} and \cite[Sec.~II.5.1]{Tem83}, the mapping $u \mapsto |Du|_K(O)$ is convex for every~$O \in \mathfrak{B}(\Omega)$ since it can be expressed as a Fenchel conjugate. Finally, we point out that decomposability of a set~$E$ can be directly formulated in terms of operations on the normal $\nu_E = D\1_E/|D\1_E|$ on $\pa E$, see e.g. \cite[Thm.~16.3]{Mag12}.

\begin{thm} \label{thm:extremalsscalar}
Let~$n=1$. Then there holds
    \begin{align*}
        \Ext(\mathcal{B}_K)= \left\{\, \left\lbrack \frac{\1_E}{\TV_K(\1_E)} \right \rbrack\;\middle\vert\; E\subset \Omega \  \text{simple}, \ |E| \in \big(0,|\Omega|\big)\right\}.
    \end{align*}
\end{thm}
\begin{proof}
The proof is similar to the isotropic case in~\cite[Theorem~4.7]{BreCar20} but makes crucial use of the representation in~\eqref{eq:TVKfromnormals} as well as the associated coarea formula.
    We first show that
    \begin{align*}
     \left\{\,\left\lbrack\frac{\1_E}{\TV_K(\1_E)}\right \rbrack\;\middle\vert\; E\subset \Omega \  \text{simple}, \ |E| \in \big(0,|\Omega|\big)\right\} \subset   \Ext(\mathcal{B}_K).
    \end{align*}
    For this purpose, let~$E \subset \Omega$ be simple and denote by~$E^\du$ and~$E^\dz$ its measure theoretic interior and exterior, respectively. Now, let~$\lambda \in (0,1)$ as well as~$u_1, u_2 \in \BV(\Omega)$ with~$\TV_K(u_1)\leq 1, \TV_K(u_2)\leq 1$ be such that there is~$c \in \R$ with
    \begin{align} \label{eq:convcomb}
        \frac{\1_E}{\TV_K(\1_E)}+c=\lambda u_1+ (1-\lambda) u_2 \quad \text{and thus} \quad \frac{D\1_E}{\TV_K(\1_E)}=\lambda Du_1+ (1-\lambda) Du_2. 
    \end{align}
    Recall that the mapping~$u \mapsto |Du|_K (O) $ is convex, i.e., we have 
    \begin{align} \label{eq:convaux}
                \frac{|D\1_E|_K(O)}{\TV_K(\1_E)} \leq\lambda |Du_1|_K(O)+ (1-\lambda) |Du_2|_K(O) \quad \text{for all} \quad O \in \mathfrak{B}(\Omega).
    \end{align}
    Now, we want to conclude that in fact there holds
    \begin{align*}
         \frac{|D\1_E|_K(O)}{\TV_K(\1_E)} =\lambda |Du_1|_K(O)+ (1-\lambda) |Du_2|_K(O) \quad \text{for all} \quad O \in \mathfrak{B}(\Omega).
    \end{align*}
Indeed, assume that there is~$O \in \mathfrak{B}(\Omega) $ such that the inequality in~\eqref{eq:convaux} is strict. We estimate
\begin{align*}
    1=  \frac{|D\1_E|_K(\Omega)}{\TV_K(\1_E)}&=\frac{|D\1_E|_K(O)}{\TV_K(\1_E)}+\frac{|D\1_E|_K(\Omega \setminus O)}{\TV_K(\1_E)} \\&\leq \frac{|D\1_E|_K(O)}{\TV_K(\1_E)}+ \lambda |Du_1|_K(\Omega \setminus O)
    +(1- \lambda) |Du_2|_K(\Omega \setminus O) \\
    &< \lambda |Du_1|_K(\Omega)
    +(1- \lambda) |Du_2|_K(\Omega)\leq 1,
\end{align*}
yielding a contradiction. In particular, we have~$\TV_K(u_i)=|Du_i|_K(\Omega)=1, i=1,2$. Moreover, using \eqref{eq:D1E-Ha} we conclude
\begin{align*}
0&=|D\1_E|_K(E^\du)=|D\1_E|_K(E^\dz)\\&=|Du_1|_K(E^\du)=|Du_2|_K(E^\dz)=|Du_1|_K(E^\du)=|Du_2|_K(E^\dz).
\end{align*}
Note that there is~$c_K>0$ with
\begin{align} \label{eq:equivnorms}
    c_K |Du_i|(O) \leq |Du_i|_K (O) \quad \text{for all} \quad O \in \mathfrak{B}(\Omega) \quad \text{and} \quad   i=1,2.
\end{align}
Hence, invoking~\cite[Lemma 4.6]{BreCar20}, we conclude
\begin{align*}
    \lbrack u_1 \rbrack = \lbrack d_1 \1_E \rbrack, \quad  u_2= \lbrack d_2 \1_E \rbrack \quad \text{where} \quad |d_1|=|d_2|=\TV_K(\1_E)^{-1}.
\end{align*}
Note that we necessarily have~$d_1=d_2=\TV_K(\1_E)^{-1}$. Indeed,~$d_1$ and~$d_2$ cannot be both negative, due to~\eqref{eq:convcomb}, and mixed signs are not possible due to~$|2\lambda-1|<1$. As a consequence, we conclude
\begin{align*}
    \lbrack u_1 \rbrack= \lbrack u_2 \rbrack= \left\lbrack \frac{\1_E}{\TV_K(\1_E)}\right \rbrack,
\end{align*}
which yields the extremality of~$\lbrack \1_E/\TV_K(\1_E) \rbrack $.
For the converse inclusion, let~$\lbrack u \rbrack$ be extremal and define the function
\begin{align*}
    G(s)= \int_{-\infty}^{s} \int_\Omega \left|\frac{D\1_{\{u > t\}}}{|D\1_{\{u > t\}}|}(x)\right|_K \mathrm{d}|D\1_{\{u > t\}}|(x) \mathrm{d} t \quad \text{with} \quad \lim_{s \rightarrow -\infty} G(s)=0, \quad \lim_{s \rightarrow +\infty} G(s)=1.
\end{align*}
Since~$G$ is continuous, there is~$\Bar{s} \in \R$ with~$G(\Bar{s})=1/2$. Consequently, we have
\begin{align*}
    \lbrack u \rbrack =\frac{1}{2} \lbrack u_1 \rbrack+ \frac{1}{2} \lbrack u_2 \rbrack \quad \text{where} \quad u_1= 2 \min\{u, \bar s\}, \quad u_2= 2 \max\{u-\bar s,0\}.
\end{align*}
Since~$\lbrack u \rbrack$ is extremal, this implies that~$u=u_1+c_1=u_2+c_2$ for some~$c_1,c_2 \in \R$. Thus, for a.e.~$x \in \Omega$, we conclude~$u(x)=2\bar s+c_1$ if~$u(x) \geq \bar s$ and~$u(x)=c_2$ else. Since~$\TV_K(u)=1$, i.e.~$Du \neq 0$, this implies that $u$ achieves exactly two values almost everywhere on~$\Omega$. W.l.o.g, possibly by a change of representative, we can assume~$u(x)\in \{0,a\}$ for a.e.~$x\in\Omega$ and some~$a>0$. Again noting that~$\TV_K(u)=1$, we arrive at
\begin{align*}
    \lbrack u \rbrack = \left\lbrack \frac{\1_E}{\TV_K(\1_E)} \right \rbrack  \quad \text{where} \quad E \coloneqq \left\{\,x \in \Omega\;\middle\vert\;u(x)=a\right\}.
\end{align*}
Note that the latter is uniquely defined up to sets of Lebesgue-measure zero.
It remains to show that~$E$ is simple, i.e. both $E$ and~$\Omega \setminus E$ are indecomposable. For this purpose, first assume that~$E$ is decomposable, i.e., there are $A, B \subset \Omega$ with~$E=A \cup B$,~$|A|,|B|>0$, $|A \cap B|=0$ and~$\Per(E,\Omega)=\Per(A,\Omega)+\Per(B,\Omega)$. We note that this implies~$\Ha^{d-1}(\partial^* A \cap \partial^* B )=0$ as well as~$D\1_E=D \1_A+D \1_B$. This can be derived in the following way:
Using Federer's theorem \cite[Thm.~16.2]{Mag12}, De Giorgi's structure theorem \cite[Thm.~15.9]{Mag12} and~\cite[Theorem 16.3, (16.12)]{Mag12}, we get
\begin{align*}
\Per(A, B^\du)+\Per(A, B^\dz)&+ \Per(B, A^\du)+\Per(B, A^\dz)+ 2 \Ha^{d-1}(\partial^* A \cap \partial^* B )\\&=\Per(E,\Omega)=\Per(A, \Omega)+\Per(B, \Omega),
\end{align*}
as well as
\begin{align*}
    \Per(E,\Omega) &= \Per(A,B^\dz)+\Per(B,A^\dz)+ \Ha^{d-1} \left( \left\{x \in \partial^* A \cap \partial^* B \;\middle\vert\; \nu_A(x)=\nu_B(x)\,\right\} \right)
    \\ & \leq
    \Per(A,B^\dz)+\Per(B,A^\dz)+ \Ha^{d-1} \left(  \partial^* A \cap \partial^* B \right).
\end{align*}
due to~$E=A \cup B$. Plugging the second estimate into the first one and rearranging yields
\begin{align*}
    0 \leq \Per(A, B^\du)+ \Per(B, A^\du)+  \Ha^{d-1}(\partial^* A \cap \partial^* B ) \leq 0
\end{align*}
and, consequently,
\begin{align*}
 \Ha^{d-1}(\partial^* A \cap \partial^* B )=\Per(A,B^\du)=\Per(B,A^\du)=0.
\end{align*}
Finally,~\cite[Thm. 16.3, (16.6)]{Mag12} yields~$Du=D \1_A \mres B^\dz+D \1_B \mres A^\dz=D \1_A+D \1_B$. In summary, we thus have
\begin{align*}
    \TV_K(\1_E) &= \int_{\partial^* E} \left| \nu_E (x)\right|_K~\mathrm{d} \Ha^{d-1}(x)=\int_{\partial^* A} \left| \nu_A (x)\right|_K~\mathrm{d} \Ha^{d-1}(x)+\int_{\partial^* B} \left| \nu_B (x)\right|_K~\mathrm{d} \Ha^{d-1}(x) \\
    &= \TV_K(\1_A)+\TV_K(\1_B),
\end{align*}
where the second equality makes use of~$\Ha^{d-1}(\partial^* A \cap \partial^* B )=0$. Since~$|A|, |B|>0$ but~$|A \cap B|=0$, we have that~$\1_A$ and~$\1_B$ are not constant. This implies $\TV_K(\1_A),\TV_K(\1_B) >0$, so we can write
\begin{align*}
    u_1= \frac{\1_A}{\TV_K(\1_A)}, \quad u_2= \frac{\1_B}{\TV_K(\1_B)} \quad \text{noting that} \quad u= \frac{\TV_K(\1_A)}{\TV_K(\1_E)} u_1+ \frac{\TV_K(\1_B)}{\TV_K(\1_E)} u_2.
\end{align*}
Hence, this yields a nontrivial convex combination of~$[u]$ contradicting its extremality. As a consequence,~$E$ needs to be indecomposable. Similar arguments as well as noting that~$\TV_K(\1_E)=\TV_K(\1_{\Omega \setminus E})$ yield the indecomposability of~$\Omega \setminus E$ and, finally, that~$E$ is simple.
\end{proof}

\section{Additive matrix norms} \label{sec:additive}
We now turn to the general case of vector-valued functions of total variation on multidimensional domains (i.e.~with $d \geq 2$) together with specific choices of the matrix norm. As a first example, we consider a norm of the form \eqref{eq:KsKvNorm} where $|\cdot|_{K_v}=|\cdot|_{1}$ and $K_s$ is arbitrary. In this case, we observe that
\begin{align} \label{eq:groupnorm}
    \TV_K(u)=\sum^n_{j=1}\TV_{K_s}(u_j) \quad \text{for all} \quad u=(u_1,\dots, u_n) \in \BV(\Omega; \R^n). 
\end{align}
Hence, due to the following abstract lemma, the characterization of~$\Ext(\mathcal{B}_K)$ can be traced back to the scalar case. 
\begin{lemma}\label{lem:additivenorms}
Let $V$ be a vector space that decomposes in a direct sum as 
\[V = V_1 \oplus \ldots \oplus V_n,\]
and $f: V \to \R \cup \{+\infty\}$ a positively one-homogeneous convex functional such that $f(v)=0$ only if $v=0$, and
\begin{equation}\label{eq:additivity}f(v)=\sum_{i=1}^n f(v_i) \ \text{ whenever }\ v=\sum_{i=1}^n v_i \ \text{ with }\ v_i \in V_i.\end{equation}
Then there holds
\begin{align} \label{eq:extofsum}
v=\sum_{i=1}^n v_i \in \Ext\big(\{ w \in V \,|\, f(w) \leq 1\}\big)
\end{align}
if and only if there is an index $i_0 \in \{1,\ldots,n\}$ as well as~$v_{i_0}$ with
\begin{align} \label{eq:extofpart}
v=v_{i_0} \quad \text{as well as} \quad v_{i_0} \in \Ext\big(\{ w \in V_{i_0} \,|\, f(w) \leq 1\}\big).
\end{align}
\end{lemma}
\begin{proof}
We start by proving~$\eqref{eq:extofsum} \Rightarrow \eqref{eq:extofpart}$. For this purpose, we consider the two cases~$v=0$ and~$v \neq 0$ separately. Assume that~$v=0$ is an extremal point of $\{ w \in V \,|\, f(w) \leq 1\}$ but~\eqref{eq:extofpart} does not hold. Since~$V$ is a direct sum, this also implies~$v_i=0 \in V_i$ for all $i=1,\dots, n$. Select an arbitrary index~$\bar{\imath}$. Since  \eqref{eq:extofpart} does not hold and~$v=v_{\bar{\imath}}=0$, we conclude that~$0 \in V_{\bar{\imath}}$ is not an extremal point of~$\{ w \in V_{\bar{\imath}} \,|\, f(w) \leq 1\}$. Consequently, there are $u,w \in V_{\bar{\imath}},~u \neq w$ and~$\lambda \in (0,1)$ with 
\begin{align*}
  f(u), f(w) \in (0,1] \quad \text{as well as} \quad   0=(1-\lambda)u+ \lambda w \in V_{\bar{\imath}}.
\end{align*}
Since $V_{\bar{\imath}} \subset V$, this yields a contradiction to the extremality of~$0$ in $\{ w \in V \,|\, f(w) \leq 1\}$. We assume now that $v \neq 0$ is extremal in $\{ w \in V \,|\, f(w) \leq 1\}$, but there is more than one index $i$ with $f(v_i) \neq 0$. Then, by partitioning $\{1,\ldots,n\}$ into two sets so that each contains one such index, and grouping the $v_i$ accordingly, we can write 
\begin{equation}\begin{aligned}v&=u+w \quad\text{with }u \text{ and }w\text{ linearly independent,}\\f(v)&=f(u)+f(w)=1 \text{ and }f(u),f(w) \in (0,1).\end{aligned}\end{equation}
But then,
\[v = f(u) \frac{u}{f(u)} + f(w) \frac{w}{f(w)}\]
is a nontrivial convex combination for $v$ with elements of $\{ z \in V \,|\, f(z) = 1\}$, contradicting the assumed extremality of $v$. Since
\begin{align*}
    f(v) = \sum^n_{i=1} f(v_i)
\end{align*}
and~$f$ only vanishes at zero, this contradiction and the assumption $v \neq 0$ shows that there is precisely one index $i_0$ for which $f(v_{i_0}) \neq 0$ and we have~$v=v_{i_0}$. Moreover, if we had $v=v_{i_0} = \lambda u_{i_0} + (1-\lambda) w_{i_0}$ with $u_{i_0}, w_{i_0} \in \{v \in V_{i_0} \,\vert\, f(v) \leq 1\}$, we can use the extremality of $v$ in $\{ w \in V \,|\, f(w) \leq 1\}$ to infer that $u_{i_0}=w_{i_0}=v_{i_0}$, concluding that $v_{i_0}$ must be extremal in $\{v \in V_{i_0} \,\vert\, f(v) \leq 1\}$.

For the converse, if $v=v_{i_0} \in \Ext(\{ z \in V_{i_0} \,|\, f(z) \leq 1\})$, we first note that~$f(v)=f(v_{i_0})=1$ since~$v_{i_0} \neq 0$ is an extremal point and~$f$ only vanishes at zero. Now assume that we can write 
\begin{gather*}v_{i_0} = \lambda u + (1-\lambda) w \quad \text{with } u,w \in \{ z \in V \,|\, f(z) \leq 1\}, \text{ }\lambda \in (0,1), \text{ and}\\u=\sum_{i=1}^n u_i,\quad w=\sum_{i=1}^n w_i \quad\text{with }u_i, w_i \in V_i.\end{gather*}
Since~$V$ decomposes into the direct sum of the~$V_i$ this implies
\begin{align*}
    v_{i_0}=\lambda u_{i_0} + (1-\lambda) w_{i_0} \quad \text{as well as } \ 0= \lambda u_{i} + (1-\lambda) w_{i} \ \text{ for all }i \neq i_0.
\end{align*}
We estimate
\begin{align*}
    1=f(v_{i_0}) \leq \lambda f(u_{i_0}) + (1-\lambda) f(w_{i_0}) \leq  \sum^n_{i=1} \left\lbrack \lambda f(u_{i}) + (1-\lambda) f(w_{i}) \right \rbrack= \lambda f(u)+ (1-\lambda) f(w) \leq 1. 
\end{align*}
Here, the second inequality follows from
\begin{align*}
    0=f(0)=f(\lambda u_{i} + (1-\lambda) w_{i}) \leq \lambda f( u_{i}) + (1-\lambda) f( w_{i})  \quad \text{for all} \quad i \neq i_0
\end{align*}
while the final equality follows from the linearity of~$f$ w.r.t the direct sum decomposition. As a consequence, we conclude~$f(u_i)=f(w_i)=0$ and, since~$f$ only vanishes at zero,~$u_i=w_i=0$ for all~$i \neq i_0$. Consequently, we have~$u=u_{i_0} \in V_{i_0}$,~$w=w_{i_0} \in V_{i_0}$ as well as~$f(u_{i_0}) \leq 1,~f(w_{i_0}) \leq 1$. Extremality of~$v_{i_0}$ in~$\Ext(\{ z \in V_{i_0} \,|\, f(z) \leq 1\})$ finally yields~$u=w=v_{i_0}$, i.e.~$v=v_{i_0}$ is an extremal point of~$\Ext(\{ z \in V \,|\, f(z) \leq 1\})$.
\end{proof}

\begin{cor}\label{cor:TV12ext}
Let~$|\cdot|_K$ be given by~\eqref{eq:groupnorm}. Then there holds
\begin{align*}
    \Ext(\mathcal{B}_K)= \left\{\,\left\lbrack e_j\frac{\1_E}{\TV_{K_{s}}(\1_E)} \right\rbrack\;\middle\vert\; E \subset \Omega~\text{simple}, \ |E| \in \big(0,|\Omega|\big),~j=1,\dots,n \,\right\},
\end{align*}
where~$\{e_i\}^n_{i=1}$ denotes the canonical basis of~$\R^n$. 
\end{cor}
\begin{proof}
   The claimed statement follows from Lemma~\ref{lem:additivenorms} together with Theorem~\ref{thm:extremalsscalar}, setting
   \begin{align*}
       V= \BV(\Omega;\R^n)/(\R^n \1_\Omega), \quad V_i= \big\{ \, [(u \cdot e_j)e_j]\;|\; u \in \BV(\Omega;\R^n) \big\}
   \end{align*}
   as well as~$f([u])= \TV_K(u)$ which satisfies
   \[f([u])=\sum^n_{j=1} \TV_{K_s}(u \cdot e_j)=\sum^n_{j=1} \TV_{K}\big((u \cdot e_j)e_j\big)= \sum^n_{j=1} f\big(\big[(u \cdot e_j)e_j\big]\big).\qedhere\]
\end{proof}
\begin{rem}
The work \cite{ParNow22} explores some extensions to deep networks of the Radon transform based approach to shallow ReLU networks developed in \cite{OngWilSouSre20, ParNow21}. The latter is centered on the seminorm
\[|f|_{\mathscr{R} \BV^2(\R^d)} = \mathscr{R} \TV^2 (f) = c_d \|\partial_t^2 \Lambda^{d-1} \mathscr{R} f\|_{\M(\S^{d-1} \times \R)},\]
where $c_d=2(2\pi)^{d-1}$, $\mathscr{R}:\R^d \to \S^{d-1} \times \R$ is the Radon transform, and $\Lambda^{d-1}$ is the ramp filter occurring in its inversion formula by backprojection, given by
\[\Lambda^{d-1} g(\theta, t) = \begin{cases}\partial_t^{d-1} g(\theta, t) &\text{ if } d-1 \text{ even,} \\ \mathscr{H}_t \partial_t^{d-1} g(\theta, t) &\text{ if } d-1 \text{ odd,}\end{cases}\] 
with $\mathscr{H}_t$ the Hilbert transform in the shift variable $t$.

The natural extension to deep networks considered in \cite{ParNow22} consists in the learning problem
\begin{equation}\label{eq:deeplearning}\min_{\substack{  f^{(1)},\ldots,f^{(L)}\\ f^{(\ell)} \in \mathscr{R} \BV^2(\R^{d_{\ell - 1}}; \R^{d_\ell}) \\ f = f^{(L)} \circ \ldots \circ f^{(1)}}} \sum_{n=1}^N \mathscr{L}(y_n, f(x_n)) + \alpha \sum_{\ell = 1}^L \left| f^{(\ell)} \right|_{\mathscr{R} \BV^2(\R^{d_{\ell - 1}}; \R^{d_\ell})}\end{equation}
where $\mathscr{R} \BV^2(\R^{d_{\ell - 1}}; \R^{d_\ell})$ is a Cartesian product of $d_\ell$ copies of $\mathscr{R} \BV^2(\R^{d_{\ell-1}})$ with norm
\[\left\| g \right\|_{\mathscr{R} \BV^2(\R^{d_{\ell - 1}}; \R^{d_\ell})} = \sum_{m=1}^{d_\ell} \left\| g_m \right\|_{\mathscr{R} \BV^2(\R^{d_{\ell - 1}})}\]
and $\mathscr{L}:\R^{d_L} \times \R^{d_0} \to \R$ is some loss function. For this problem, in \cite[Thm.~3.2]{ParNow22} the authors claim a representer theorem by a recursion argument layer by layer, leading to solutions corresponding to an architecture with at most $N^L (d_1 \cdot \ldots \cdot d_L)$ units.

Given the Cartesian product structure and additiveness of the regularizer over components and layers, we can apply Lemma \ref{lem:additivenorms} to find that the extremal points of the unit ball induced by the seminorm
\[\big(f^{(1)},\ldots, f^{(L)}\big) \mapsto \sum_{\ell = 1}^L \left| f^{(\ell)} \right|_{\mathscr{R} \BV^2(\R^{d_{\ell - 1}}; \R^{d_\ell})}\]
consist on the equivalence classes a single ReLU unit with an affine skip connection (arising from the kernel of $\mathscr{R} \TV^2$, see \cite[Thm.~1 and Rem.~2]{ParNow21}) in one layer and an affine function on all the other ones. However, from this fact we cannot immediately infer anything about \eqref{eq:deeplearning}, because the appearance of the composition $f = f^{(L)} \circ \ldots \circ f^{(1)}$ in the data fitting term makes the problem strongly nonconvex.
\end{rem}

\section{Large families of extremals for matrix norms with symmetries} \label{sec:gennorm}
 In this section, we consider more general matrix norms in which~$|\cdot|_K$ does not have additivity properties allowing to decompose~$\BV(\Omega;\R^n)$ into a direct sum of subspaces, so that Lemma~\ref{lem:additivenorms} and Corollary~\ref{cor:TV12ext} are not applicable. We divide our focus between extremal points for the unit ball $\mathcal{B}_K$ associated to $\TV_K$ as in \eqref{eq:defBK}, and the analogous ball for $\TD_K$ in $\BD(\Omega)$, that is
 \[\mathcal{D}_K := \big\{ [u] \in \BD(\Omega)/\mathcal{A} \;\big\vert\; \TD_K(u) \leq 1 \big\}.\]
 For $\TV_K$ and in the interest of covering commonly used norms, this section puts the focus on classes of matrix norms which are invariant w.r.t to certain orthogonal transformations. More in detail, we call~$|\cdot|_K$~\textit{left orthogonally invariant} if~$|QA|_K=|A|_K$ for all~$Q \in O(n)$ and~$A \in \R^{n \times d}$ while~$|\cdot|_K$ is~\textit{isotropic} if $|AR|_K = |A|_K$ for all rotations $R \in \SO(d)$. Note that Schatten~$p$ norms with $p \in [1,\infty]$, e.g. the nuclear, Frobenius and the spectral norm, as well as the Ky Fan norms are, both, isotropic and left orthogonally invariant.
 
 We provide three main results: 
 First, piecewise constant functions supported on simple sets still represent extremal points of~$\Ext(\mathcal{B}_K)$, assuming $|\cdot|_K$ satisfies either 
 \begin{equation}
 \label{eq:clunkyround}\begin{gathered}|\cdot|_K \text{ is left orthogonally invariant, and } \\
 |e_1 \cdot b||e_1 \otimes \nu|_K = |(e_1 \otimes e_1)(b \otimes \nu)|_K < |b \otimes \nu|_K \\
 \text{ for all }b \in \R^n,\nu \in \R^d \text{ with }|b|=1, |\nu|=1 \text{ and } b \neq e_1,\end{gathered}\end{equation}
 or 
 \begin{equation}
 \label{eq:rankoneiso}|b \otimes \nu|_K = |b|\ \ \text{ for all }b \in \R^n\text{ and }\nu \in \R^d\text{ with }|\nu| = 1,
 \end{equation}
 where $e_1$ denotes the first canonical basis vector in $\R^n$. 

Second, these characteristic functions \emph{do not} yield a full characterization of~$\Ext(\mathcal{B}_K)$ under the second assumption. In this case, we give an infinite family of extremal points which attain more than two values on~$\Omega$. We remark that, on the one hand, norms of the form \eqref{eq:KvKsNorm} with $|\cdot|_{K_v} = |\cdot|$ and $K_s$ arbitrary satisfy \eqref{eq:clunkyround} but not  \eqref{eq:rankoneiso} unless $|\cdot|_{K_s} = |\cdot|$. On the other hand, all Schatten and Ky Fan norms satisfy \eqref{eq:rankoneiso}, but for certain nonsmooth cases like the spectral norm and the $1$-Schatten norm, \eqref{eq:clunkyround} fails. A further remark concerning \eqref{eq:clunkyround} is that orthogonal invariance implies $|(e_1 \otimes e_1)A|_K \leq |A|_K$, since any contraction (and in particular $e_1 \otimes e_1$) is a convex combination of matrices in $O(n)$, see \cite{SauParWil15}. However, the strict inequality is a genuine additional assumption.

Third, we prove that equivalence classes $[w\1_E]$ of piecewise infinitesimally rigid vector fields with two regions are also generically extremal in $\mathcal{D}_K$. We require strict convexity of $|\cdot|_K$ with respect to certain rank-two variations, and the boundary of $E$ to not be completely flat, with this second condition shown to be sharp with a family of counterexamples.
 
Similar to the scalar-valued case, see \eqref{eq:TVKfromnormals}, it will be useful to have a primal expression for~$\TV_K(u)$ and $\TD_K(u)$:
\begin{lemma}\label{lem:TVKprimal}
We have
\begin{equation}\label{eq:TVKprimal}\TV_K(u) = \int_\Omega \left|\frac{Du}{|Du|_F}\right|_K \dd|Du|_F.\end{equation}
\end{lemma}
Analogously, we also have
\begin{equation}\label{eq:TDKprimal}\TD_K(u) = \int_\Omega \left|\frac{\E u}{|\E u|_F}\right|_K \dd|\E u|_F.\end{equation}
\begin{proof}
The proof of \eqref{eq:TVKprimal} follows completely analogously to \cite[Prop.~9]{AmbAziBreUns24}, whose main steps we indicate here for convenience. One can apply the Lusin theorem to the $\Rnd$-valued function
\[M := \frac{Du}{|Du|_F} \left| \frac{Du}{|Du|_F}\right|^{-1}_K,\]
obtaining for each $\e$ a set $\Omega_\e$ on which $M$ is continuous and the integral of the measure inside \eqref{eq:TVKprimal} on $\Omega \setminus \Omega_\e$ is less than $\e$. This continuity allows to find functions $N_\e$ satisfying $M \cdot N_\e \geq 1- \e$ on $\Omega_\e$ while taking only finitely many values and satisfying the constraint $|\cdot|_{\Kc} \leq 1$ pointwise, which are then mollified to make them admissible in the definition of $\TV_K$. The proof of \eqref{eq:TDKprimal} is completely analogous.
\end{proof}

The next step is to realize that if we can express a piecewise constant function~$u_0$ as a nontrivial convex combination of functions with the same value of $\TV_K$, then those functions need to be themselves piecewise constant:
\begin{lemma}\label{lem:pwccomb}
Let $u_0 \in \BV(\Omega; \R^n)$ be piecewise constant, and expressible as a nontrivial convex combination
\begin{equation}\label{eq:convcombgame}u_0 = \lambda v + \big(1 - \lambda\big)w, \quad \text{ with } \TV_K(v)=\TV_K(w)=\TV_K(u_0) \text{ and }\lambda \in (0,1).\end{equation}
Then $v$ and $w$ must also be piecewise constant and $J_v \cup J_w = J_{u_0}$ (mod $\Ha^{d-1}$).
\end{lemma}
\begin{proof}
Without loss of generality we can take $\TV_K(u_0)=1$. We have the decomposition \cite[Sec.~3.9]{AmbFusPal00} into absolutely continuous, jump and Cantor parts for $u_0$:
\begin{equation}\label{eq:decompder}Du_0 = \nabla u_0 \mathcal{L}^d + D^j u_0 + D^c u_0 = \nabla u_0 \mathcal{L}^d +\big[(u_0^+ - u_0^-) \otimes \nu_{u_0}\big] \Ha^{d-1} \mres J_{u_0} + D^c u_0,\end{equation}
and similarly for $v$ and $w$. Now, by the definition of piecewise constant, $\nabla u_0$ and $D^c u_0$ vanish. In view of the expression \eqref{eq:convcombgame} as convex combination and by testing on sets of Hausdorff dimension $d-1$, on which the Lebesgue measure $\mathcal{L}^d$ and the Cantor parts $D^c v$ and $D^c w$ all vanish \cite[Prop.~3.92]{AmbFusPal00}, we must have
\begin{gather*}\big[(u_0^+ - u_0^-) \otimes \nu_{u_0}\big] \Ha^{d-1}\mres J_{u_0} = \lambda D^j v + (1-\lambda) D^j w \\ \text{with }\int_{J_{u_0}} \big|(u_0^+ - u_0^-) \otimes \nu_{u_0}\big|_K \dd \Ha^{d-1} = 1,\end{gather*}
which together with
\begin{align*}
    1 &= \int_\Omega |\nabla v|_K \dd x+ |D^c v|_K(\Omega) + |D^j v|_K(J_{u_0})+ |D^j v|_K(\Omega \setminus J_{u_0})  \\
    &= \int_\Omega |\nabla w|_K \dd x + |D^c w|_K(\Omega) + |D^j w|_K(J_{u_0})+ |D^j w|_K(\Omega \setminus J_{u_0})
\end{align*}
implies that we must have 
\[|D^j v|_K(J_{u_0})=|D^j w|_K(J_{u_0})=1,\] 
as well as
\begin{gather*}\nabla v = 0, \ D^c v = 0, \ D^j v \mres (J_v \setminus J_{u_0})=0, \\  \nabla w = 0, \ D^c w = 0 \text{ and } D^j w \mres (J_w \setminus J_{u_0})=0,\end{gather*}
since otherwise we would necessarily have $\TV_K(v) > 1$ and $\TV_K(w) >1$. This implies that $v$ and $w$ must be piecewise constant using \cite[Thm.~4.23]{AmbFusPal00}, since even though this result is stated for scalar-valued functions in $L^\infty$, we can just apply it to truncations of components $(C \wedge v^i ) \vee (-C), (C \wedge w^i ) \vee (-C)$ for all $i = 1,\ldots, n$ and all $C>0$ to obtain (as observed in the proof of \cite[Lem.~5.1]{DalToa22}) that $v^i,w^i$ must be piecewise constant, implying that $v,w$ themselves are also piecewise constant. Moreover, by the expression of $u_0$ as a convex combination we must have $J_{u_0} \subset J_v \cup J_w$ (mod $\Ha^{d-1}$). On the other hand, we deduce that we must also have $J_v \cup J_w \subset J_{u_0}$ (mod $\Ha^{d-1}$), since~$D^j v \mres (J_v \setminus J_{u_0})$ and~$D^j w \mres (J_w \setminus J_{u_0})$ vanish, respectively.

From this, we conclude $J_v \cup J_w = J_{u_0}$ (mod $\Ha^{d-1}$).
\end{proof}

\begin{lemma}\label{lem:pwccombBD}
Let $u_0 \in \BD(\Omega)$ be piecewise infinitesimally rigid and such that there is a nontrivial convex combination
\begin{equation}\label{eq:convcombBDgame}u_0 = \lambda v + \big(1 - \lambda\big)w, \quad \text{ with } \TD_K(v)=\TD_K(w)=\TD_K(u_0) \text{ and }\lambda \in (0,1).\end{equation}
Then $v$ and $w$ must be piecewise infinitesimally rigid and $J_v \cup J_w = J_{u_0}$ (mod $\Ha^{d-1}$).
\end{lemma}
\begin{proof}
For $u \in \BD(\Omega)$, one has \cite[Sec.~4]{AmbCosDal97} the decomposition analogous to \eqref{eq:decompder}:
\[\begin{aligned}\E u &= \frac12\big(\nabla u + (\nabla u)^\top\big) \mathcal{L}^d + \E^j u + \E^c u \\
&= \frac12\big(\nabla u + (\nabla u)^\top\big) \mathcal{L}^d +\big[(u^+ - u^-) \odot \nu_u\big] \Ha^{d-1} \mres J_{u} + \E^c u.\end{aligned}\]
Moreover, by \cite[Thm.~A.1]{ChaGiaPon07} functions $u \in \BD(\Omega)$ for which $\E^c u = 0$ and $Eu := \nabla u + (\nabla u)^\top = 0$ must be piecewise infinitesimally rigid. Using these results, one can follow an analogous proof as for Lemma \ref{lem:pwccomb}.
\end{proof}
 
\begin{thm}\label{thm:indicatrixextremals}
Let~$|\cdot|_K$ satisfy either \eqref{eq:clunkyround} or \eqref{eq:rankoneiso}, and define the vector norm
\begin{align*}
    |\nu|_k=|e_1 \otimes \nu |_K \quad \text{for all} \quad \nu \in \R^d, 
\end{align*}
where $e_1 \in \R^n$ is the first canonical basis vector. Then there holds
\begin{align}\label{eq:uEb}
    \left\{\,\left\lbrack \frac{1}{\TV_k(\1_E)} b\1_{E} \right \rbrack\;\middle\vert\;E \subset \Omega~\text{simple},\ |E| \in \big(0,|\Omega|\big),~b \in \R^n,~|b|_k=1\,\right\} \subset \Ext(\mathcal{B}_K).
\end{align}

\end{thm}
\begin{proof}
If $|\cdot|_K$ satisfies \eqref{eq:rankoneiso} then we immediately deduce that $|\cdot|_k = |\cdot|$ and $|b \otimes \nu_E|_K=|b|=1$ for $\nu_E = D\1_E/|D\1_E|$, implying that for
\[u_{E,b} := \frac{1}{\TV_k(\1_E)} b\1_{E} \ \text{ with }\ Du_{E,b} = \frac{1}{\TV_k(\1_E)} \big(b \otimes \nu_E\big) \Ha^{d-1} \mres \pa E\]
and $E,b$ as in \eqref{eq:uEb}, we have by using \eqref{eq:TVKprimal} that
\[\TV_K(u_{E,b}) = \frac{1}{\TV_k(\1_E)} \int_{\pa E} |b \otimes \nu_E|_K \dd \Ha^{d-1} = \frac{1}{\TV_k(\1_E)}\Ha^{d-1}\big(\pa E\big) = \frac{\Per(E, \Omega)}{\Per(E,\Omega)} = 1.\]
Moreover, we claim that $\TV_K(u_{E,b})=1$ also holds if \eqref{eq:clunkyround} is assumed. To see this, first we show that the left orthogonal invariance of~$|\cdot|_K$ implies~$\TV_K(Qu)=\TV_K(u)$ for all~$u \in \BV(\Omega;\R^n)$,~$Q \in O(n)$ and where left multiplication has to be understood pointwise. Indeed, recalling~\eqref{eq:Kduality}, we get
\begin{align*}
    |QA|_{\Kc} &= \sup \left\{ \tr\big(A^\top Q^\top B\big)\,\middle\vert\, B \in \Rnd, \ |B|_K \ls 1\right\}\\&= \sup \left\{ \tr\big(A^\top Q^\top B\big)\,\middle\vert\, B \in \Rnd, \ |Q^\top B|_K \ls 1\right\} \\&=
    \sup \left\{ \tr\big(A^\top  \Tilde{B}\big)\,\middle\vert\, \Tilde{B} \in \Rnd, \ |\tilde B|_K \ls 1\right\}= |A|_{\Kc}
\end{align*}
for all~$A \in \Rnd$ and~$Q \in O(n)$. Hence,~$|\cdot|_{\Kc}$ is left orthogonally invariant as well. This implies
\begin{align*}
\TV_K(Qu) &= \sup \left\{ \int_\Omega (Qu) \cdot \div \Phi ~\mathrm{d}x \,\middle\vert\,  \Phi \in C^1_c(\Omega;\Rnd), \, |\Phi(x)|_{\Kc} \leq 1 \text{ for all } x\in \Omega \right\} \\
&= \sup \left\{ \int_\Omega u \cdot \div Q^\top \Phi~\mathrm{d}x \,\middle\vert\,  \Phi \in C^1_c(\Omega;\Rnd), \, |Q^\top \Phi(x)|_{\Kc} \leq 1 \text{ for all } x\in \Omega \right\}
\\
&= \sup \left\{ \int_\Omega u \cdot \div \tilde \Phi~\mathrm{d}x \,\middle\vert\,  \tilde \Phi \in C^1_c(\Omega;\Rnd), \, |\tilde \Phi(x)|_{\Kc} \leq 1 \text{ for all } x\in \Omega \right\}=\TV_K(u).
\end{align*}
Now, let~$u_{E,b} = b\1_E / \TV_k(\1_E)$ for $b \in \R^n$ with $|b|_k=1$. To verify that $\TV_K(u_{E,b})=1$, we denote by $R_{b}\in \SO(n)$ any rotation transforming $b$ to a vector parallel to $e_1$. We have then that~$\TV_K(u_{E,b})=\TV_K(R_b u_{E,b}) = \TV_K(u_{E,e_1})$, to the right hand side of which we can apply \eqref{eq:TVKprimal} and \eqref{eq:TVKfromnormals} to get
\[\TV_K(u_{E,e_1}) = \frac{1}{\TV_k(\1_E)} \int_{\pa E} |e_1 \otimes \nu_E|_K \dd \Ha^{d-1} = \int_{\pa E} |\nu_E|_k \dd \Ha^{d-1} = \frac{\TV_k(\1_E)}{\TV_k(\1_E)}=1.\]
To arrive at extremality of $u_{E,b}$, assume that we have 
\begin{equation}\label{eq:indcomb}u_{E,b} = \lambda u_1 + (1-\lambda) u_2 \quad\text{with }\TV_K(u_1)=\TV_K(u_2)=1 \text{ and }\lambda \in (0,1).\end{equation}
In this situation, we can apply Lemma \ref{lem:pwccomb} to conclude that $u_1,u_2$ are piecewise constant with $J_{u_1} \cup J_{u_2}=\pa E$ (mod $\Ha^{d-1}$). Now by the Federer theorem \cite[Thm.~16.2]{Mag12} we have $\Ha^{d-1}(\pe E \setminus \pa E)=0$ for $\pe E = \Omega \setminus (E^\du \cup E^\dz)$ and since the $u_i$ are piecewise constant, we have that $D u_i$ is absolutely continuous with respect to $\Ha^{d-1} \mres \pa E$ for $i=1,2$. In particular, this implies that for each component $u_i^j$ with $i=1,2$ and $j=1,\ldots,n$ we have $|Du_i^j|(E^\du)=0$ and $|Du_i^j|(E^\dz)=0$, and since these sets are indecomposable because $E$ is assumed simple, we can apply a constancy theorem for scalar $\BV$ functions (see \cite[Prop.~2.12]{DolMul95} or \cite[Rem.~2]{AmbCasMasMor01}) to conclude that
\[u_i = u_{E,b_i} + a_i \quad\text{with } a_i,b_i \in \R^n \quad\text{for }i=1,2,\]
for which by the invariance with respect to addition of constant functions (i.e., taking adequate representatives) we may assume $a_1 = a_2 = 0$. Moreover, for these we have
\[1=\TV_K(u_i) = \frac{|b_i|_k}{\TV_k(\1_E)} \int_{\partial^\ast E} \left| \frac{b_i}{|b_i|_k}\right|_K \dd \Ha^{d-1} = |b_i|_k,\]
in addition to
\begin{equation}\label{eq:TVintjumpconv}1=\TV_K(u_{E,b})=\frac{1}{\TV_k(\1_E)}\int_{\pa E} \big|\big( \lambda b_1 + (1-\lambda) b_2 \big) \otimes \nu_E(x) \big|_K \dd \Ha^{d-1}(x).\end{equation}
In case \eqref{eq:rankoneiso} is satisfied, the integral in the right hand side of \eqref{eq:TVintjumpconv} above reduces to 
\[|\lambda b_1 + (1-\lambda) b_2| \, \Ha^{d-1}(\pa E),\] and we conclude directly by noting that all unit vectors are extremal for the $\ell^2$ ball in $\R^n$, so that $b_1=b_2=b$. 

If instead we have \eqref{eq:clunkyround}, we argue by contradiction and assume that at least one of the $b_i$ is not parallel to $b$, which without loss of generality we can assume to be~$b_1$. Letting~$R_b$ as above denote a rotation that maps~$b$ to a vector parallel to~$e_1$, we then have~$R_b b_1 \not \in \Span(e_1)$. We compose $u_{E,b}$ on the left with the pointwise projection $\pi_{b}a=(a \cdot b) b / |b|^2$ onto the line $\R b \subset \R^n$, which leaves $u_{E,b}$ invariant. Noticing that 
\[\pi_{b} = R_b^\top \pi_{e_1} R_b = \frac{b}{|b|} \otimes \frac{b}{|b|},\] 
using the left orthogonal invariance, the triangle inequality, \eqref{eq:clunkyround} combined with $\Ha^{d-1}(\partial^\ast E) > 0$, and again left orthogonal invariance,
\[\begin{aligned}1&=\TV_K(u_{E,b})=\TV_K( \pi_{b} u_{E,b})=\TV_K( \pi_b u_{E,b})\\
&=\frac{1}{\TV_k(\1_E)}\int_{\pa E} \big|(e_1 \otimes e_1)\left[\big( \lambda R_b b_1 + (1-\lambda) R_b b_2 \big) \otimes \nu_E(x) \right] \big|_K \dd \Ha^{d-1}(x)\\
& \leq \frac{1}{\TV_k(\1_E)} \Bigg( \lambda \int_{\pa E} \big| (e_1 \otimes e_1) \big(R_b b_1 \otimes \nu_E(x)\big) \big|_K \dd \Ha^{d-1}(x) 
\\&\qquad\qquad\qquad + (1-\lambda)\int_{\pa E} \big| (e_1 \otimes e_1) \big(R_b b_2 \otimes \nu_E(x)\big) \big|_K \dd \Ha^{d-1}(x) \Bigg)\\
& < \frac{1}{\TV_k(\1_E)} \Bigg( \lambda \int_{\pa E} \big|\big(R_b b_1 \otimes \nu_E(x)\big) \big|_K \dd \Ha^{d-1}(x) 
\\&\qquad\qquad\qquad + (1-\lambda)\int_{\pa E} \big|\big(R_b b_2 \otimes \nu_E(x)\big) \big|_K \dd \Ha^{d-1}(x) \Bigg)\\
& = \lambda \TV_K(R_b u_{E,b_1}) + (1-\lambda) \TV_K(R_b u_{E,b_2}) \\
& = \lambda \TV_K(u_{E,b_1}) + (1-\lambda) \TV_K(u_{E,b_2})= 1,\end{aligned}\]
which is indeed a contradiction. We have obtained that $b_1$ and $b_2$ must both be parallel to $b$, so that we can now apply the scalar-valued anisotropic case of Theorem \ref{thm:extremalsscalar} on $\TV_k$ to conclude.
\end{proof}

We now turn our attention to vector fields in $\BD(\Omega)$, for which in particular $n=d$, and where additional difficulties appear owing to the symmetrized gradient having a larger kernel, the infinitesimal rigid motions appearing in \eqref{eq:infrigid}. Specifically, we prove that piecewise infinitesimally rigid functions with two regions (the analog of indicatrices in the $\TV$ case) arising from a simple finite perimeter set $E$ whose boundaries are not completely flat, also give rise to extremal points of sublevel sets of $\TD_K$ in $\BD(\Omega)/\mathcal{A}$. To profit from this assumption, we will use that by De Giorgi's structure theorem \cite[Thm.~3.59]{AmbFusPal00} $\pa E$ is countably $(d-1)-$rectifiable, which in turn implies \cite[Thm.~2.76]{AmbFusPal00} that there are countably many Lipschitz $(d-1)-$graphs $\Gamma_\ell$ such that
\begin{equation}\label{eq:rectifiability}\pa E = F \cup \bigcup_{\ell=1}^\infty \Gamma_\ell \quad \text{with }\Ha^{d-1}(F)=0, \ \text{ and }\Ha^{d-1}(\Gamma_\ell)>0 \ \text{ for all }\ell \in \N.\end{equation}

\begin{thm}\label{prop:indicatrixextremalsBD}
Assume that $d \geq 2$, and that for all $a,b \in \R^d$ linearly independent and $\nu \in \R^d$ with $|\nu|=1$, we have that the function
\begin{equation}\label{eq:symrankoneconv}
(0,1) \ni \lambda \mapsto \big|\big((1-\lambda)a + \lambda b\big) \odot \nu\big|_K
\end{equation}
is strictly convex. Let $E \subset \Omega$ be a simple set with $|E| \in \big(0,|\Omega|\big)$, and
\begin{equation}\label{eq:notinplane}
\Ha^{d-1}\big( \pa E \setminus \{ x \in \Omega \,\vert\, a \cdot x = c \}\big) >0 \quad\text{for all }\, a \in \R^d\setminus\{0\} \text{ and } c \in \R.\end{equation}
Then, the equivalence classes of functions of the form $u_{E,w} := w \1_E$ for $w \in \mathcal{A}$ are extremal in 
\[ \big\{ [u] \in \BD(\Omega)/\mathcal{A} \;\big\vert\; \TD_K(u) \leq \TD_K(u_{E,w})\big\}.\]
\end{thm}
\begin{proof}
By homogeneity we can assume that $\TD_K(u_{E,w})=1$, and as for Theorem \ref{thm:indicatrixextremals}, assume for a contradiction that we can write
\begin{equation}\label{eq:indcombBD}u_{E,w} = \lambda u_1 + (1-\lambda) u_2 \quad\text{with }u_1 \neq u_2,\ \TD_K(u_1)=\TD_K(u_2)=1 \text{ and }\lambda \in (0,1).\end{equation}
Then, using Lemma \ref{lem:pwccombBD} we have that $u_1, u_2$ must be piecewise infinitesimally rigid. In this setting, the constancy theorem is not applicable, since it's proved using the coarea formula, but we can take advantage of \cite[Thm.~A.1]{ChaGiaPon07} being formulated on arbitrary open sets to conclude that $u_1$ and $u_2$ are affine on both $E^\du$ and $E^\dz$. Therefore, we can write
\[u_i = u_{E,w_i} + v_i \quad \text{with }w_i,v_i \in \mathcal{A},\]
and as before pick representatives such that $v_i=0$, which in turn forces $w_i \neq 0$. We then have, using \eqref{eq:TDKprimal}, that
\begin{equation}\label{eq:TDintjumpconv}\begin{aligned}
1=\TD_K(u_{E,w})&=\int_{\pa E} |w(x) \odot \nu_E(x)|_K \dd \Ha^{d-1}(x)\\
&=\int_{\pa E} \big|\big( \lambda w_1(x) + (1-\lambda) w_2(x) \big) \odot \nu_E(x) \big|_K \dd \Ha^{d-1}(x)
\end{aligned}\end{equation}
in which, in contrast to the situation in \eqref{eq:TVintjumpconv}, the functions $w_i$ and $w$ are of course not necessarily constant. We can assume that the $w_i$ are linearly independent as functions, since otherwise we immediately reach $w_1=w_2$ and consequently, the contradiction $u_1=u_2$. 

Moreover, expressing the reduced boundary $\pa E$ as in \eqref{eq:rectifiability}, if we find any point $x_0 \in \pa E \setminus F$ for which the vectors $w_1(x_0), w_2(x_0)$ are linearly independent, then by continuity the functions $w_1,w_2$ also remain linearly independent in $B(x_0, r)$ for some $r>0$, while $x_0 \in \Gamma_\ell$ for some $\ell$ implies
\begin{equation}\label{eq:linindepoften}\Ha^{d-1}\big(\pa E \cap B(x_0, r)\big) > 0.\end{equation}
Further, for any point $x \in \pa E \cap B(x_0, r)$, using the strict convexity property in \eqref{eq:symrankoneconv} we can directly estimate
\begin{equation}\label{eq:integrandstrictconv}\begin{aligned}
|w(x) \odot \nu_E(x)|_K &= \big|\big( \lambda w_1(x) + (1-\lambda) w_2(x) \big) \odot \nu_E(x) \big|_K\\
&< \lambda |w_1(x) \odot \nu_E(x)|_K + (1-\lambda) |w_2(x) \odot \nu_E(x)|_K,
\end{aligned}\end{equation}
which when used back in \eqref{eq:TDintjumpconv} and taking into account \eqref{eq:linindepoften} gives 
\[1=\TD_K(u_{E,w}) < \lambda \TD_K(u_{E,w_1}) + (1-\lambda) \TD_K(u_{E,w_2}) = \lambda + (1-\lambda) = 1,\]
the desired contradiction.

It remains to show that $x \mapsto w_i(x) = A_i x + b_i \in \mathcal{A}$,~$w_i \neq 0$, being independent as functions implies the existence of such an $x_0$. For this, let us assume that for every point $x \in \pa E \setminus F$ the vectors $w_1(x), w_2(x)$ are linearly dependent. A nonzero skew-symmetric matrix must necessarily have rank at least $2$, since if we take $a,b \in \R^d$, skew-symmetry of $b \otimes a$ would imply that for every $i \in \{1, \ldots, d\}$ we have $a_i b_i = 0$, which then gives that all off-diagonal elements $b_i a_j = - b_j a_i$ with $i \neq j$ must vanish as well. This means, since it cannot happen that $A_2=0$ and $b_2 =0$ simultaneously, that defining
\[G := \big\{x \in \R^d \,\big\vert\, A_2 x + b_2 = 0\big\}, \text{ we have }\Ha^{d-1}(G)=0,\]
and by linear dependence we can find a function $c: \pa E \setminus (F \cup G) \to \R$ such that
\begin{equation}\label{eq:funnylindep}A_1 x + b_1 = c(x) \big[A_2 x + b_2\big].\end{equation}
Now, we notice that translating $\Omega$ and shifting the functions $w_i$ accordingly does not affect the linear structure of $\BD(\Omega)$, so extremality and linear independence are preserved (even if the particular values in $A_i$ and $b_i$ change). This allows us, taking into account that $\Ha^{d-1}(G)=0$ but $\Ha^{d-1}(\pa E \setminus F) > 0 $, to assume that $0 \in \pa E \setminus ( F \cup G )$ to simplify our computations. In particular we then have $b_1 = c(0) b_2$, so we can write
\begin{equation}\label{eq:funnylindepspec}A_1 x = c(x) A_2 x + \big[c(x)-c(0)\big]b_2.\end{equation}
To proceed further, we now make a case distinction:

\textbf{Case 1} Assume that~$b_2 \neq 0$. In this case,~\eqref{eq:notinplane} ensures that 
\[\Ha^{d-1}(\pa E \setminus ( F \cup G \cup \{\,x\;|\;x^\top b_2=0\,\} ))>0.\]
Now, we select linearly independent~$x_1,\dots,x_d \in \R^d $ with
\begin{align*}
    x_1 \in \pa E \setminus ( F \cup G \cup \{\,x\;|\;x^\top b_2=0\,\} )
\end{align*}
as well as
\begin{align*}
    x_{i+1} \in \pa E \setminus ( F \cup G \cup \{\,x\;|\;x^\top b_2=0\,\} \cup \Span\{x_1,\ldots,x_i\}  )
\end{align*}
for~$i=1,\dots, d-1$. This construction is again well-defined due to~\eqref{eq:notinplane}, since for $i=1,\dots,d-1$ we have that $\Span\{x_1,\ldots,x_i\}$ is contained in a hyperplane. Using the skew-symmetry of $A_1$, ~\eqref{eq:funnylindepspec}, and the skew-symmetry of~$A_2$ yields
\begin{align*}
    0= x^\top_i A_1 x_i= c(x_i) x^\top_i A_2 x_i+\big[c(x_i)-c(0)\big]\big(x_i^\top b_2\big)=\big[c(x_i)-c(0)\big]\big(x_i^\top b_2\big),   
\end{align*}
and~$c(x_i)=c(0)$,~$i=1,\dots,d$. Thus, evaluating~\eqref{eq:funnylindepspec} again for every~$x_i$, we get~$A_1 x_i= c(0)A_2 x_i$ for every~$i=1,\dots,d$ and, since~$\{x_1,\dots,x_d\}$ is a basis of~$\R^d$, finally~$A_1=c(0) A_2$. Because we already saw that~$b_1=c(0)b_2$, this means $w_1$ and~$w_2$ are linearly dependent as functions.

\textbf{Case 2} If~$b_1=b_2=0$, using~\eqref{eq:notinplane} and with similar arguments as in Case 1, we sequentially select a family $\{x_1,\dots,x_d\}$ with~$\Span\{x_1,\ldots,x_d\}=\R^d$ and~$x^\top_i A_2 x_{i+1} \neq 0$ for all~$i=1,\dots,d-1$. Using~\eqref{eq:funnylindepspec}, skew-symmetry of~$A_1$, again~\eqref{eq:funnylindepspec}, and finally skew-symmetry of~$A_2$, we get
\begin{align*}
  c(x_{i+1}) x^\top_{i} A_2 x_{i+1}= x^\top_{i} A_1 x_{i+1}=-x^\top_{i+1} A_1 x_{i}= -c(x_{i}) x^\top_{i+1} A_2 x_{i}=c(x_{i}) x^\top_{i} A_2 x_{i+1}      
\end{align*}
and thus by construction~$c(x_1)=\cdots=c(x_d)$. Denoting this common constant by~$c$, we get~$A_1 x_i= c A_2 x_i$ for~$i=1,\dots,d$ implying that~$A_1= cA_2$.
\end{proof}

\begin{example}
We demonstrate with an explicit counterexample that the condition \eqref{eq:notinplane}, expressing that the boundary of $E$ should not be completely flat, cannot be dropped. Let
\[\Omega = (-1,1)^2, \quad E = \{(x_1,x_2) \in \Omega \,\vert\, x_2 < 0\},\quad\text{ and }\ u_{E,e_2} = e_2 \1_E.\]
Further, let us define a function $z: \Omega \to \R^2$ as
\[z(x)=\begin{cases} A x &\text{ if }x_2 \geq 0\\ -A x &\text{ if }x_2 < 0,\end{cases} \quad \text{ where }\ A=\begin{pmatrix} 0 & 1 \\ -1 & 0 \end{pmatrix},\]
and finally
\[u_1 := u_{E,e_2} + \frac{1}{4}z, \quad u_2 := u_{E,e_2} - \frac{1}{4}z, \quad \text{ so that }\ u_{E,e_2} = \frac{1}{2}u_1 + \frac{1}{2}u_2.\]
Then, we can directly compute
\begin{equation}\begin{aligned}
\TD_K(u_{E,e_2}) &= \int_{(-1,1) \times \{0\}} |e_2 \otimes e_2|_K\dd \Ha^1(x),\\
\TD_K(u_1) &= \int_{(-1,1) \times \{0\}} \left(1+\frac{x_1}{2}\right) |e_2 \otimes e_2|_K\dd \Ha^1(x),\\
\TD_K(u_2) &= \int_{(-1,1) \times \{0\}} \left(1-\frac{x_1}{2}\right) |e_2 \otimes e_2|_K\dd \Ha^1(x)
\end{aligned}\end{equation}
and by symmetry of the domain of integration all of these are equal, from which we conclude that $u_{E,e_2}$ is not extremal. We remark that this does not depend on the particular value of $|e_2 \otimes e_2|_K$ or any other properties of $|\cdot|_K$. In particular it is straightforward to check, by applying adequate rotations and translations to extensions by zero of $u_1$ and $u_2$, that for any bounded domain $\Omega \subset \R^d$ with $d \geq 2$ and any halfspace $H = \{ x \in \R^d \, \vert \, x \cdot \nu < 0\}$, $\nu \in \S^{d-1}$ with $H \cap \Omega \neq \emptyset$, the function $u_{H,\nu}$ is not extremal.
\end{example}

\begin{rem}
It could be tempting to attempt to relax \eqref{eq:symrankoneconv} into strict convexity of both of the functions
 \begin{align}
 \label{eq:rankoneconv1} (0,1) \ni \lambda &\mapsto \big|\big((1-\lambda)a + \lambda b\big) \otimes \nu\big|_K\ \ \text{ and }\\
 \label{eq:rankoneconv2} (0,1) \ni \lambda &\mapsto \big|\nu \otimes \big((1-\lambda)a + \lambda b\big)\big|_K,
 \end{align}
but this would not allow us to conclude, since in that case we would have to use the triangle inequality, which would in turn prevent recovering $|w_i \odot \nu_E|$ in the right hand side of \eqref{eq:integrandstrictconv}. On the other hand, using the positive one-homogeneity of $|\cdot|_K$ it is straightforward to check that the strict convexity in \eqref{eq:rankoneconv1} is less restrictive than assuming \eqref{eq:rankoneiso}, and one could prove Theorem \ref{thm:indicatrixextremals} assuming the former. We have chosen to focus on \eqref{eq:clunkyround} and \eqref{eq:rankoneiso} because in those cases we can relate the total variation of a function with two values with the (possibly anisotropic) perimeter induced by the norm $|\cdot|_k$ on $\R^d$, which makes the similarities and differences with scalar-valued functions more transparent. In any case, let us note that \eqref{eq:clunkyround} and \eqref{eq:rankoneiso} are still formulated around rank-one matrices, so it does not seem to be possible to deduce \eqref{eq:symrankoneconv} from them either.
\end{rem}

Returning to the total variation, we know that for the scalar-valued case, sets that are not simple cannot give rise to extremal points by adjusting the function values on different indecomposable components to find nontrivial convex combinations. However, in the vectorial case for norms satisfying \eqref{eq:rankoneiso}, such a construction is not possible and extremals may have more than two values. As a (non-exhaustive) example of this phenomenon, we write out explicit assumptions to find extremals with three values in the next result.

\begin{thm}\label{thm:thehouseawayswins}
Let $|\cdot|_K$ satisfy \eqref{eq:rankoneiso}. Assume that $b_1, b_2 \in \R^n$ with $b_1,b_2 \neq 0$ are not collinear and $E_1, E_2$ are two simple disjoint sets with $|E_1|,|E_2| \in \big(0,|\Omega|\big)$,
\begin{equation}\label{eq:bdyparts}\begin{gathered}
\mu_1 := \Per\big(E_1, E_2^{(0)}\cap \Omega\big) >0, \quad \mu_2 := \Per\big(E_2, E_1^{(0)} \cap \Omega\big) >0\text{ and}
\\ \mu_- := \Ha^{d-1}\left(\big\{x \in \pa E_1 \cap \pa E_2 \,\big\vert\, \nu_{E_1}(x) = -\nu_{E_2}(x)\big\}\cap \Omega \right) >0.\end{gathered}\end{equation}
Then, the equivalence class of the function $u_0=b_1 \1_{E_1} + b_2 \1_{E_2}$ is extremal in \[\big[\TV_K(u_0)\big] \mathcal{B}_K = \big\{[u] \in \BV(\Omega; \R^n)/(\R^n \1_\Omega) \,\big|\, \TV_K(u) \ls \TV_K(u_0)\big\}.\]
\end{thm}
\begin{proof}
Assuming a nontrivial decomposition $u_0 = \lambda v + (1-\lambda)w$ as in \eqref{eq:convcombgame}, we can apply Lemma \ref{lem:pwccomb} to see that both $v$ and $w$ are piecewise constant and \[J_v \cup J_w = J_{u_0} = \pa E_1 \cup \pa E_2  \quad (\text{mod }\Ha^{d-1}).\]
Therefore, as in the proof of Theorem \ref{thm:indicatrixextremals}, by taking adequate representatives for the quotient we can write
\[v = b^v_1 \1_{E_1} + b^v_2 \1_{E_2} \text{ and } w = b^w_1 \1_{E_1} + b^w_2 \1_{E_2},\]
with 
\begin{equation}\label{eq:totallength}\begin{aligned}
&\mu_1 |b_1| + \mu_2 |b_2| + \mu_{-} |b_1-b_2| \\&\quad= \mu_1 |b^v_1| + \mu_2 |b^v_2| + \mu_{-} |b^v_1-b^v_2| = \mu_1 |b^w_1| + \mu_2 |b^w_2| + \mu_{-} |b^w_1-b^w_2| = 1
\end{aligned}\end{equation}
for $\mu_1,\mu_2,\mu_-$ as in \eqref{eq:bdyparts}, and where we have used \cite[Thm.~16.3]{Mag12} for expressing the different boundary terms, the assumptions on $E_1, E_2$, and the property \eqref{eq:rankoneiso} assumed for the norm, so that 
\[|(b_1 - b_2) \otimes \nu_{E_1}|_K = |b_1 - b_2|\] 
independently of the direction of $\nu_{E_1}$, and similarly for the other two terms.

Moreover, by the expression as convex combination \eqref{eq:convcombgame} we know that
\begin{equation}\label{eq:vwconvcomb}\begin{gathered}b_1 = \lambda b^v_1 + (1-\lambda) b_1^w, \quad b_2 = \lambda b^v_2 + (1-\lambda) b_2^w \text{ and }\\ (b_1 - b_2) = \lambda (b^v_1 - b^v_2) + (1-\lambda) (b_1^w - b_2^w),\end{gathered}\end{equation}
implying
\begin{equation}\label{eq:vwconvcombnorms}\begin{gathered}|b_1| \ls \lambda |b^v_1| + (1-\lambda) |b^w_1|, \quad |b_2| \ls \lambda |b^v_2| + (1-\lambda) |b^w_2| \text{ and }\\ |b_1 - b_2| \ls \lambda |b^v_1 - b^v_2| + (1-\lambda) |b_1^w - b_2^w|,\end{gathered}\end{equation}
and by strict convexity of $|\cdot|$, these inequalities are strict unless the two vectors in the corresponding convex combination (and in consequence, the result) are multiples of each other. But if any of them were not an equality we would contradict \eqref{eq:totallength}, so we have for some $r, s \in \R$ that 
that $b^v_1 =  r b_1$ and $b^v_2 = s b_2$, which using the first two equalities of \eqref{eq:vwconvcomb} gives rise to
\[b_1^w = \frac{1 - \lambda r}{1- \lambda}b_1, \quad \text{and} \quad b_2^w = \frac{1 - \lambda s}{1- \lambda}b_2,\]
so that the third equality of \eqref{eq:vwconvcomb} becomes
\[b_1 - b_ 2 = \lambda \big( rb_1 - sb_2) + (1-\lambda) \left(  \frac{1 - \lambda r}{1- \lambda}b_1 - \frac{1 - \lambda s}{1- \lambda}b_2\right).\]
But for the second line of \eqref{eq:vwconvcombnorms} to hold with equality, again to avoid a contradiction with \eqref{eq:totallength}, we must have that the three points appearing above must also be multiples of each other. In particular $rb_1 - sb_2$ must be parallel to $b_1-b_2$, which since $b_1$ and $b_2$ are not collinear, forces $r=s$. But then, plugging in the representations for $b_1^v,b_2^v,b_1^w,b_2^w$ into \eqref{eq:totallength}, we see that we must have simultaneously
\[|r|=1, \ \text{ and }\ \left|\frac{1-\lambda r}{1-\lambda}\right|,\]
which forces $r=1$.
\end{proof}

\begin{rem}
We note that \cite{AmbAziBreUns24} is also centered on Schatten norms, and their Proposition 23 has a similar flavor, in that continuous piecewise affine functions (in their second derivative setting) are extremal for $\TV^2$ if and only if they are with respect to variations of the coefficients while keeping the jump set of the first derivative fixed. Here we make the observation that for the total variation of vector valued functions we end up with more extremals than what one would at first expect. This stands in contrast to the scalar case, the separable case of Lemma \ref{lem:additivenorms}, or the situation considered in \cite{AmbAziBreUns24}. In this last case, continuity of the piecewise affine functions plays a crucial role, since it imposes that the difference of the gradients on either side of an interface must be proportional to the normal to it.
\end{rem}

\begin{rem}\label{rem:polyhedralvalues}
In Theorem \ref{thm:thehouseawayswins} we have assumed that for rank-one matrices the matrix norm used satisfies $|b \otimes \nu|_k = |b|$. In contrast, for a norm leading to the sum of the componentwise variations, see Corollary \ref{cor:TV12ext}, one does not observe the same phenomenon in which functions with more than two values could be extremal. It is reasonable to ask what the boundary is between these two situations, and specifically whether these more complicated extremals can appear when using a norm of the type \eqref{eq:KsKvNorm} with $K_s$ the Euclidean ball in $\R^d$ and $K_v$ a polyhedron in $\R^n$. To see that they can, notice that the proof of Theorem \ref{thm:thehouseawayswins} only used strict convexity of the vector norm at the vectors $b_1, b_2$ and $b_1 - b_2$, so it is also applicable as long as those three vectors are extremal points of a rescaling of $K_v$. A straightforward example is to consider $n=2$ and $K_v$ the regular octagon with vertices 
\[y_j := \left(\cos\left(\frac{2\pi j}{8}\right), \sin\left(\frac{2\pi j}{8}\right)\right)^\top\] 
for $j = 0,\ldots,7$. In this case, for any $j$ the difference $y_{j} - y_{j-2}$ (with the indexing counted modulo $8$) is proportional to a vertex of the octagon, and hence extremal in a ball of $|\cdot|_{K_v}$.
\end{rem}

One might also wonder whether results analogous to Theorem \ref{thm:thehouseawayswins} also hold for the case of $\TD_K$. The presence of the symmetrized gradient makes the situation more involved because of possible cancellations and because it precludes the simplifications afforded by the condition \eqref{eq:rankoneiso}, so the given proof is not applicable to that situation. We think that the existence of piecewise infinitesimally rigid extremals with more than two regions is nevertheless quite plausible but do not pursue this further, due to the increased complexity of the ensuing computations coupled with the more limited interest for such a result.

\section{Unit radial vector fields}\label{sec:hedgehog}

In \cite{AmbBreCon23} it is proved that a function whose graph is a spherically symmetric cone is an extremal for the Schatten $\TV^2$ ball. Below we show that its derivative (a `hedgehog' vector field) is also extremal for the total variation defined using the Frobenius norm, when $n=d=2$. Since we only treat this case, in this section we denote it just by $\TV$.

\begin{prop}\label{prop:hedgehogTV}
Let $\Omega = B(0,1) \subset \R^2$, $n=2$, and $\TV$ be defined with the Frobenius norm. Then the equivalence class of
\begin{align*}
    u_H(x):=\begin{cases}\frac{x}{|x|} & \text{ if } x \neq (0,0) \\ (0,0 ) & \text{ if } x = (0,0)\end{cases}
\end{align*}
is extremal in \[\big\{ [u] \in \BV(\Omega; \R^2)/(\R^2 \1_{\Omega}) \,|\, \TV(u) \ls \TV(u_H)\big\}.\]
\end{prop}
\begin{proof}
We first notice that indeed $u_H \in \BV(\Omega; \R^2)$, and assume that it can be expressed as 
\[u_H = \lambda v + (1-\lambda) w \quad\text{with }\TV(v)=\TV(w)=\TV(u_H)\ \text{ and }\lambda \in (0,1).\]
Note that we can, w.l.o.g, assume that~$Dv$ and~$Dw$ are absolutely continuous w.r.t the Lebesgue measure, that is $v,w \in W^{1,1}(\Omega; \R^d)$ and
\[\TV(v) = \int_\Omega |\nabla v(x)|_F \dd x, \quad \TV(w) = \int_\Omega |\nabla w(x)|_F \dd x.\]
First, we argue that we necessarily have
\begin{align} \label{eq:eqofvardiff}
    |\nabla u_H(x) |_F= \big| \lambda \nabla v(x) + (1-\lambda) \nabla w(x) \big|_F=\lambda | \nabla v(x) |_F+ (1-\lambda) | \nabla w(x)|_F 
\end{align}
for a.e.~$x \in \Omega$. Indeed, by convexity we know that
\begin{align*}
     \big| \lambda \nabla v(x) + (1-\lambda) \nabla w(x) \big|_F \leq \lambda | \nabla v(x) |_F+ (1-\lambda) | \nabla w(x)|_F \quad \text{for a.e.}~x \in \Omega. 
\end{align*}
Assume that this inequality is strict on a set~$O \in \mathfrak{B}(\Omega)$ with $\L^2(O)>0$. Together with the assumption~$\TV(u_H)=\TV(v)=\TV(w)$, we then have
\begin{align*}
    \TV(u_H) &= \int_O \big| \lambda \nabla v(x) + (1-\lambda) \nabla w(x)\big|_F~\dd x+ \int_{\Omega \setminus O} \big| \lambda \nabla v(x) + (1-\lambda) \nabla w(x)\big|_F~\dd x \\
    & < \int_\Omega \big\lbrack \lambda |\nabla v(x)|_F+ (1-\lambda) |\nabla w(x)|_F \big\rbrack~\dd x = \lambda \TV(v)+ (1-\lambda)\TV(w)=\TV(u_H)
\end{align*}
yielding a contradiction.

We denote the polar coordinate parametrization of $B(0,1)$ by $\Upsilon: (0,1)\ \times [0, 2\pi) \to B(0,1)$ 
and by
\begin{align*}
    e_r(x)=u_H(x), \quad e_\varphi(x)= (x_2,- x_1)/|x| \quad \text{for}~x \in B(0,1) \setminus \{0\}
\end{align*}
the associated unit vectors. We point out that the vectors  $\{e_r(x), e_\varphi(x)\} $ form an orthonormal basis of $\R^2$ while the matrices
\begin{equation}\label{eq:matrixframe}\big\{e_r(x) \otimes e_r(x),  e_r(x) \otimes e_{\varphi}(x), e_{\varphi}(x) \otimes e_r(x), e_{\varphi}(x) \otimes e_{\varphi}(x) \big\}\end{equation}
are an orthonormal basis of~$\R^{2 \times 2}$ equipped with the Frobenius inner product for all~$x \in B(0,1) \setminus  \{0\}$. For a generic function~$u \in W^{1,1}(B(0,1);\R^2)$ with derivative~$Du= \nabla u \ \L^2$, this implies the pointwise representations
\begin{equation*}
    u= (u \cdot e_\varphi) e_\varphi+ (u \cdot e_r) e_r \quad \L^2-\text{a.e.}
\end{equation*}
as well as
\begin{equation}\label{eq:polarderiv}\begin{aligned}
 \nabla u &= \big(e_r^\top \nabla u e_r\big)(e_r \otimes e_r) + \big(e_{\varphi}^\top \nabla u e_{\varphi}\big)(e_{\varphi} \otimes e_{\varphi}) \\
 & \quad+ \big(e_{\varphi}^\top \nabla u e_r\big)(e_{\varphi} \otimes e_r) + \big(e_r^\top \nabla u e_{\varphi}\big)(e_r \otimes e_{\varphi}) \quad \L^2-\text{a.e.} 
\end{aligned}\end{equation}
and thus
\begin{equation}\label{eq:polarderivnorm}
 |\nabla u|^2_F=\big(e_r^\top \nabla u e_r\big)^2+ \big(e^\top_{\varphi} \nabla u e_{\varphi}\big)^2+ \big(e^\top_{\varphi} \nabla u e_{r}\big)^2+ \big(e^\top_{r} \nabla u e_{{\varphi}}\big)^2 \quad \L^2-\text{a.e.}
\end{equation}
However, while we have $e_r, e_{\varphi} \in C^\infty(B(0,1)\setminus \{0\}\,;\, \S^1)$, we note that~$e_r \in W^{1,1}(B(0,1),\R^2) $ satisfies
\begin{equation}\label{eq:polarhedhehog}  De_r = Du_H= \nabla e_r \mathcal{L}^2 = \left[ x \mapsto \frac{1}{|x|}\left(\Id - \frac{x}{|x|} \otimes \frac{x}{|x|}\right)\right]\mathcal{L}^2  = \frac{1}{|\cdot|} (e_{\varphi} \otimes e_{\varphi}) \,\mathcal{L}^2.\end{equation}
As a consequence, $u \cdot e_r$ is not guaranteed to belong to $W^{1,1}(B(0,1))$, since 
\[\nabla e_r \in L^{2-\e}(B(0,1); \R^2)\text{ for all }\e > 0, \text{ but }\nabla e_r \notin L^{2}(B(0,1); \R^2),\]
while
\[u \in W^{1,1}(B(0,1); \R^2) \subset L^{2}(B(0,1); \R^2).\]
In order to circumvent this issue, and for~$\rho \in (0,1)$ small enough, we consider the annulus $A_\rho \coloneqq B(0,1 ) \setminus \overline{B(0,\rho)}$. Note that~$e_r, e_\varphi \in C^\infty (A_\rho;\R^2)$ while the coefficient functions satisfy $(u \cdot e_r ), (u \cdot e_\varphi ) \in W^{1,1}(A_\rho)$, and we may write their gradients as
\begin{equation}\label{eq:weakdiffcoeffann}\begin{aligned}
\nabla(u \cdot e_r) &= \big(\nabla(u \cdot e_r) \cdot e_r\big)e_r + \big(\nabla(u \cdot e_r) \cdot e_\varphi \big)e_\varphi,\\
\nabla(u \cdot e_\varphi) &= \big(\nabla(u \cdot e_\varphi) \cdot e_r\big)e_r + \big(\nabla(u \cdot e_\varphi) \cdot e_\varphi \big)e_\varphi.
\end{aligned}\end{equation}

In order to conclude~$u_H=v=w$, we proceed in three steps:
\begin{itemize}
    \item[1.] We show that  that for all~$\rho \in (0, \rho_0]$,~$\rho_0$ small enough, there are $c_v(\rho), c_w(\rho) >0$ as well as~$\lambda(\rho) \in (0,1) $ 
such that the functions~$v_\rho= c_v(\rho) v$ and~$w_\rho= c_w(\rho) w$ satisfy
 \begin{align*}
\restr{u_H}{A_\rho} = \lambda(\rho) \left( \restr{v_\rho}{A_\rho} \right) + (1-\lambda(\rho)) \left( \restr{w_\rho}{A_\rho} \right), \quad |Dv_\rho|_F(A_\rho)=|Dw_\rho|_F(A_\rho)=|Du_H|_F(A_\rho).
\end{align*}
    \item[2.] By using~\eqref{eq:polarderivnorm} and~\eqref{eq:weakdiffcoeffann} as well as by passing to polar coordinates, we conclude~$u_H=v_\rho=w_\rho$ on $A_\rho$.
    \item[3.] Since the previous results are valid for all~$\rho \in (0, \rho_0]$ and $c_v(\rho),c_w(\rho) \to 1$ as $\rho \to 0$, an overlapping argument finally yields~$c_v(\rho)=c_w(\rho)=1$ for all~$\rho \in (0,\rho_0]$.
\end{itemize}

\textbf{Step 1}:
First note that by absolute continuity of $Dv$ and $Dw$ the mappings 
\begin{align*}
    \rho \mapsto |Dv|_F(A_\rho), \quad \rho \mapsto |Dw|_F(A_\rho) 
\end{align*}
are continuous on $(0,1)$ and there holds~$\lim_{\rho \searrow 0} |Dv|_F(A_\rho)=\lim_{\rho \searrow 0} |Dw|_F(A_\rho)= \TV(u_H)$. Hence, there is~$\rho_0 >0$ such that both,~$|Dv|_F(A_\rho)$ and $|Dw|_F(A_\rho)$, are nonzero for all~$\rho \in (0, \rho_0]$. From now on, let~$\rho \in (0, \rho_0]$ be arbitrary but fixed. We argue similarly to the proof of \cite[Lem.~2.12]{CarIglWal25} and consider a constructive approach for the choice of~$c_v(\rho), c_w(\rho) >0$ and~$\lambda(\rho) \in (0,1)$. Indeed, making the ansatz
\begin{align} \label{eq:rescvw}
    v_\rho := c_v(\rho) v, \quad w_\rho := c_w(\rho)w,
\end{align}
while simultaneously requiring
\begin{align} \label{eq:rescenergy}
|Dv_\rho|_F(A_\rho)= c_v(\rho) \int_{A_\rho} |\nabla v(x)|_F \dd x = c_w(\rho) \int_{A_\rho} |\nabla w(x)|_F \dd x = |Dw_\rho|_F(A_\rho)
\end{align}
as well as
\begin{align} \label{eq:restconv}
\restr{u_H}{A_\rho} = \lambda(\rho) \left( \restr{v_\rho}{A_\rho} \right) + (1-\lambda(\rho)) \left( \restr{w_\rho}{A_\rho} \right)
\end{align}
for some $c_v(\rho), c_w(\rho) >0$ and $\lambda_\rho \in (0,1)$, yields 
\begin{align} \label{eq:constequation}
    \frac{|Dv|_F(A_\rho)}{|Dw|_F(A_\rho)}= \frac{c_w(\rho)}{c_v(\rho)} \quad \text{as well as} \quad \lambda= \lambda(\rho) c_v (\rho), \quad (1-\lambda)=(1-\lambda(\rho)) c_w(\rho).
\end{align}
Here, the first identity is due to~\eqref{eq:rescenergy} while the second and third follows from a coefficient comparison between~\eqref{eq:restconv} and~$u_H=\lambda v+ (1-\lambda)w$. This implies 
\begin{align*}
    \frac{\lambda(\rho)}{1-\lambda(\rho)}= \frac{\lambda}{(1-\lambda)} \frac{c_w(\rho)}{c_v(\rho)}= \frac{\lambda}{(1-\lambda)} \frac{|Dv|_F(A_\rho)}{|Dw|_F(A_\rho)}.
\end{align*}
Rearranging yields the unique solution
\begin{align*}
    \lambda(\rho) = \frac{h(\rho)}{1+h(\rho)} \quad \text{where} \quad h(\rho)=\frac{\lambda}{(1-\lambda)} \frac{|Dv|_F(A_\rho)}{|Dw|_F(A_\rho)}>0,
\end{align*}
and thus~$\lambda(\rho) \in (0,1)$, is continuous on~$[0,\rho_0]$ as well as~$\lim_{\rho \searrow 0} \lambda(\rho)= \lambda$. We further recover continuous functions
\begin{align*}
    c_v(\rho)=\lambda/\lambda(\rho), \quad c_w(\rho)=(1-\lambda)/(1-\lambda(\rho)) \quad \text{with} \quad \lim_{\rho \searrow 0} c_v(\rho)=\lim_{\rho \searrow 0} c_w(\rho)=1. 
\end{align*}
Finally, we utilize~\eqref{eq:eqofvardiff} as well as~\eqref{eq:constequation} to conclude
\begin{align*}
    |Du_H|_F(A_\rho) &= \lambda |Dv|_F(A_\rho)+(1-\lambda) |Dw|_F(A_\rho)\\ 
    & =\lambda(\rho) c_v(\rho) |Dv|_F(A_\rho)+(1-\lambda(\rho)) c_w(\rho) |Dw|_F(A_\rho) \\
    &=\lambda(\rho)  |Dv_\rho |_F(A_\rho)+(1-\lambda(\rho)) |Dw_\rho|_F(A_\rho)=|Dv_\rho|_F(A_\rho)
\end{align*}
and thus finally~$|Du_H|_F(A_\rho)=|Dv_\rho|_F(A_\rho)=|Dw_\rho|_F(A_\rho)$.

\textbf{Step 2}: Next, we show that there holds~$u_H=v_\rho=w_\rho$ on~$A_\rho$. The proof is carried out to obtain~$u_H=v_\rho$ on~$A_\rho$, and the conclusion for~$w_\rho$ follows mutatis mutandis. For this purpose, we first prove that
\begin{equation}\label{eq:Dvisoneterm}
    \nabla v_\rho= c_v(\rho)\big(e^\top_{\varphi} \nabla v \, e_{\varphi} \big) \big(e_{\varphi} \otimes e_{\varphi} \big). 
\end{equation}
Recalling \eqref{eq:polarderivnorm} and \eqref{eq:polarhedhehog}, we obtain
\begin{align*}
|Dv_\rho|_F(A_\rho) &=|Du_H|_F(A_\rho)=\int_{A_\rho} \sqrt{ \big(e^\top_{\varphi} \nabla u_H \, e_{\varphi} \big)^2 } \dd x \\ & \leq
  \lambda(\rho) \int_{A_\rho} \sqrt{ \big(e^\top_{\varphi} \nabla v_\rho \, e_{\varphi}\big)^2 } \dd x+ (1-\lambda(\rho)) |Dw_\rho|_F(A_\rho) \leq |Dv_\rho |_F(A_\rho)
\end{align*}
where the last line follows due to~$\nabla u_H= \lambda(\rho) \nabla v_\rho + (1-\lambda(\rho)) \nabla w_\rho$,~$|Dv_\rho |_F(A_\rho)=|Dw_\rho |_F(A_\rho)=|Du_H |_F(A_\rho)$ as well as
\begin{align*}
    \int_{A_\rho} \sqrt{ \big(e^\top_{\varphi} \nabla u \,e_{\varphi}\big) ^2 } \dd x \leq \int_{A_\rho} |\nabla u|_F \dd x= |Du |_F(A_\rho) \quad \text{for all } \ u \in W^{1,1}(A_\rho,\R^2).
\end{align*}
As a consequence, 
\[\int_{A_\rho} |\nabla v_\rho|_F - |e_\varphi^\top \nabla v_\rho e_\varphi| \dd x = 0\]
and since by \eqref{eq:polarderivnorm} the integrand above is non-negative $\L^2-$a.e., we have
\[|\nabla v_\rho|_F = |e_\varphi^\top \nabla v_\rho e_\varphi| \quad \L^2-\text{a.e. in }A_\rho,\]
i.e.,~$\nabla v_\rho$ is of the form~\eqref{eq:Dvisoneterm}.

Next, we rewrite~$v_\rho = \widetilde{v}_\rho + \mathring{v}_\rho$ 
where
\begin{align*}
    \widetilde{v}_\rho= (e_r \otimes e_r)v_\rho = (v_\rho \cdot e_r)e_r, \quad \mathring{v}_\rho=(e_\varphi \otimes e_{\varphi})v_\rho= (v_\rho \cdot e_{\varphi})e_\varphi.
\end{align*}
Again recalling that~$e_r, e_\varphi \in C^\infty(A_\rho; \R^2)$ and $v_\rho \cdot e_r, v_\rho \cdot e_\varphi \in W^{1,1}(A_\rho)$, we explicitly calculate~$\nabla v_\rho= \nabla \widetilde{v}_\rho + \nabla \mathring{v}_\rho$ where 
\begin{equation}\label{eq:Dvtilde2d-new}\begin{aligned}
\nabla \widetilde{v}_\rho &= (v_\rho \cdot e_r)\nabla e_r + \big(\nabla(v_\rho \cdot e_r) \cdot e_r\big)(e_r \otimes e_r) + \big(\nabla(v_\rho \cdot e_r) \cdot e_\varphi\big)(e_r \otimes e_\varphi) \\
&= (v_\rho \cdot e_r)\left[\frac{1}{|\cdot|}(e_\varphi \otimes e_\varphi)\right] + \big(\nabla(v_\rho \cdot e_r) \cdot e_r\big)(e_r \otimes e_r) + \big(\nabla(v_\rho \cdot e_r) \cdot e_\varphi\big)(e_r \otimes e_\varphi),
\end{aligned}\end{equation}
and
\begin{equation}\label{eq:Dvring2d-new}\begin{aligned}
\nabla \mathring{v}_\rho &= (v_\rho \cdot e_\varphi)\nabla e_\varphi + \big(\nabla(v_\rho \cdot e_\varphi) \cdot e_r\big)(e_\varphi \otimes e_r) + \big(\nabla(v_\rho \cdot e_\varphi) \cdot e_\varphi\big)(e_\varphi \otimes e_\varphi) \\
&= (v_\rho \cdot e_\varphi)\left[-\frac{1}{|\cdot|}(e_r \otimes e_\varphi)\right] + \big(\nabla(v_\rho \cdot e_\varphi) \cdot e_r\big)(e_\varphi \otimes e_r) + \big(\nabla(v_\rho \cdot e_\varphi) \cdot e_\varphi\big)(e_\varphi \otimes e_\varphi).
\end{aligned}\end{equation}
Comparing coefficients with \eqref{eq:Dvisoneterm}, we conclude
\begin{align} \label{eq:importanteq}
\nabla(v_\rho \cdot e_r) e_r=\nabla(v_\rho \cdot e_\varphi) \cdot e_r=0, \quad \frac{1}{|\cdot|} (v_\rho \cdot e_\varphi)=\nabla (v_\rho \cdot e_r ) \cdot e_\varphi
\end{align} 
for a.e.~$x \in A_\rho$.

We now show that there holds~$(v_\rho \cdot e_\varphi)=0$ and $(v_\rho \cdot e_r)$ is constant in~$A_\rho$. For this purpose, consider the upper cut annulus
\begin{equation}\label{eq:defAplus}
    A^+_\rho \coloneqq \big\{\,r(\sin(\theta), \cos(\theta))\;|\;\rho <r<1,~0< \theta <\pi\,\big\}
\end{equation}
as well as
\begin{align*}
    V \colon (\rho,1) \times (0, \pi) \to \R^2, \quad V(r, \theta)= v_\rho(\Upsilon(r,\theta)).
\end{align*}
In the same way, define~$E_r= e_r \circ \Upsilon, E_\varphi= e_\varphi \circ \Upsilon $ as well as~$P_r=V \cdot E_r$ and~$P_\varphi=V \cdot E_\varphi$.
Then, since $\Upsilon$ is a bi-Lipschitz transformation between $A^+_\rho$ and $\Upsilon\big(A^+_\rho\big)$, by \cite[Thm.~2.2.2]{Zie89} we have
\begin{align*}
    V \in W^{1,1} \big((\rho,1) \times (0, \pi);\R^2\big), \quad P_r, P_\varphi \in W^{1,1} \big((\rho,1) \times (0, \pi)\big),
\end{align*}
and by the chain rule for the weak partial derivatives,~\eqref{eq:importanteq} yields
    \begin{align*}
\partial_r P_r=\partial_r P_\varphi=0, \quad \frac{1}{r} P_\varphi= \partial_\varphi P_r.
\end{align*}
Now, taking partial (distributional) derivatives w.r.t~$r$ of both sides in the second equation reveals
\begin{align*}
    \partial_r \partial_{\varphi} P_r=\partial_\varphi \partial_{r} P_r=0,
\end{align*}
using the first expression, as well as
\begin{align*}
    \partial_r \left( \frac{1}{r} P_\varphi \right)= \frac{1}{r} \partial_r P_\varphi - \frac{1}{r^2} P_\varphi=- \frac{1}{r^2} P_\varphi,
\end{align*}
noting that the chain rule for distributional derivatives is applicable since~$\lbrack (r,\varphi) \mapsto 1/r \rbrack \in C^{\infty} \big((\rho,1) \times (0, \pi)\big)$. 

From this, we conclude~$P_\varphi = \partial_\varphi P_r =\partial_r P_r = 0$ on~$A^+_\varrho$. In particular,~$P_r$ is constant. Transforming back to cartesian coordinates, this finally yields
\begin{align*}
    (v_\rho \cdot e_\varphi)=c_v(\rho)(v \cdot e_\varphi)=0  \quad \text{as well as} \quad (v \cdot e_r)= c_{\rho,+}
\end{align*}
for some~$c_{\rho,+} \in \R$
on~$A^+_\rho$. As a consequence, we arrive at~$v_\rho= c_v(\rho) v= c_v(\rho) c_{\rho,+} u_H $ on~$A^+_\rho$. Applying the same argument to the lower cut annulus $A_\rho^-$ defined analogously to \eqref{eq:defAplus} but with $\pi < \theta < 2\pi$ we obtain $v_\rho= c_v(\rho) v= c_v(\rho) c_{\rho,-} u_H $ for some~$c_{\rho,-} \in \R$ on $A_\rho^-$. Since $v_\rho \in W^{1,1}(A_\rho; \R^2)$ cannot have a jump discontinuity we immediately deduce that $c_{\rho,-} = c_{\rho,+}$. In order to avoid ambiguities, we denote this constant by~$c_\rho$. Recalling that~$|Du_H|_F(A_\rho)=|Dv_\rho|_F(A_\rho)>0$, we arrive at
\begin{align*}
    |Du_H|_F(A_\rho)=|Dv_\rho|_F(A_\rho)= c_v(\rho) |c_\rho| |Du_H|_F(A_\rho),
\end{align*}
i.e.~$c_v(\rho) |c_\rho|=1$. Hence we have~$v_\rho=u_H$ or~$v_\rho=-u_H$ on~$A_\rho$. Repeating the previous arguments for~$w_\rho$, we get~$w_\rho=u_H$ or~$w_\rho=-u_H$ on~$A_\rho$. From Step 1, we finally recall that
\begin{align*}
    \restr{u_H}{A_\rho} = \lambda(\rho) \left( \restr{v_\rho}{A_\rho} \right) + (1-\lambda(\rho)) \left( \restr{w_\rho}{A_\rho} \right)
\end{align*}
for some~$\lambda(\rho) \in (0,1)$
which implies that $v_\rho = - u_H$ or $w_\rho = - u_H$ cannot happen, i.e. we indeed have $c_v(\rho) c_\rho=1$ and~$u_H=v_\rho=w_\rho$ on~$A_\rho$.

\textbf{Step 3}: Finally, we want to conclude that~$v=w=u_H$. As before, we discuss the relevant steps for~$v$, the statement for~$w$ follows analogously. For this purpose, and for arbitrary~$\rho \in (0, \rho_0]$, we recall that we have~$u_H=v_\rho = c_v(\rho) v$ on~$A_\rho$ and thus, in particular, $v \neq 0$ on~$A_\rho$. Noting that~$A_{\rho_0} \subset A_\rho$, we conclude~$u_H=c_v(\rho) v=c_v(\rho_0) v$ on~$A_{\rho_0}$ and, consequently~$c_v(\rho)=c_v(\rho_0)$ for all~$\rho \in (0,\rho_0]$. Together with 
\[\lim_{\rho \searrow 0} c_v(\rho_0)=1\quad \text{and} \quad B(0,1)=\bigcup_{\rho \in (0, \rho_0]}A_\rho,\]
we conclude~$v=u_H$.
\end{proof}

\begin{rem}
The above proof can also be applied to $\TV_K$ defined with norms $|\cdot|_K$ different from Frobenius, since the only property of it we have used is
\[\big|\big(\nu^\top A \nu\big) (\nu \otimes \nu)\big|_K < |A|_K \ \ \text{ for all }A \in \R^{2 \times 2}\setminus \big(\R (\nu \otimes \nu)\big)\text{ and all }\nu \in \R^2\text{ with }|\nu| = 1.\]
\end{rem}

\begin{rem}\label{rem:only2d}
Naturally, one can ask whether Proposition \ref{prop:hedgehogTV} holds when $d > 2$. In this case we can denote the spherical coordinate parametrization, with the conventions for angle ordering used in \cite{Blu60}, also by 
\[\Upsilon: (0,1)\times (0, \pi)^{d-2} \times [0, 2\pi) \to B(0,1),\] 
and define the subset of the ball not covered by it as
\[\mathcal{I}_d := B(0,1) \setminus \Upsilon\Big( (0,1)\times (0, \pi)^{d-2} \times [0, 2\pi) \Big) \text{ with } \Ha^{d-1}(\mathcal{I}_d)=0,\]
to define the unit vector fields $e_r = \big(\nabla \Upsilon\big) e_1$, $e_{\varphi_i} = \big(\nabla \Upsilon\big) e_{i+1}$ with 
\[e_r, e_{\varphi_1}, \ldots, e_{\varphi_{d-1}} \in C^\infty(B(0,1)\setminus \mathcal{I}_d\,;\, \S^{d-1}).\]
In this case, the analog of \eqref{eq:Dvisoneterm}, which follows by the same orthogonality argument, becomes
\begin{equation}\label{eq:Dvisafewterms}
    \nabla v=\sum^{d-1}_{j=1} \left \lbrack  (e^\top_{\varphi_j} \nabla v e_{\varphi_j}) (e_{\varphi_j} \otimes e_{\varphi_j}) \right \rbrack  \quad \L^{d}-\text{a.e.}
\end{equation}
Then, partitioning the subset of the annulus where these vector fields are smooth as 
\begin{gather*}A_\rho \setminus \mathcal{I}_d = \Upsilon(B_1) \cup \Upsilon(B_2) \quad (\text{mod }\L^d),\ \text{ with } \\B_1 := (0,1)\times (0, \pi)^{d-2} \times (0, \pi) \text{ and }  B_2 =  (0,1)\times (0, \pi)^{d-2} \times (\pi, 2\pi),\end{gather*}
we obtain pulled back component functions
\[P_r, P_{\varphi_1}, \ldots, P_{\varphi_{d-1}} \in W^{1,1}\big(B_k \setminus \big( (0,\rho) \times \R^{d-1} \big)\big) \ \text{ with }\ \partial_r P_r = \partial_r P_{\varphi_1} = \ldots = \partial_r P_{\varphi_{d-1}} = 0.\]
Specializing to $d=3$, keeping the convention of angles above, and defining 
\[\varrho := r\sin(\varphi_1) = \sqrt{x_2^2+x_3^2},\]
we have
\[\begin{cases}\partial_{\varphi_1} e_{\varphi_1} = -\frac{1}{r} e_r, \\ \partial_{\varphi_2} e_{\varphi_1} = \frac{1}{\varrho} e_{\varphi_2}, \end{cases} \quad \text{ and }\quad \begin{cases}\partial_{\varphi_1} e_{\varphi_2} = 0, \\ \partial_{\varphi_2} e_{\varphi_2} = -\frac{1}{\varrho} e_r. \end{cases}\]
This means that, unlike when $d=2$, $\nabla e_{\varphi_1}$ has a contribution to terms in the orthonormal basis expression overlapping with the ones present in \eqref{eq:Dvisafewterms}, which prevents concluding in the same simple way as before. For this reason, we do not pursue the higher-dimensional case further. It is worth noting that the techniques used in \cite{AmbBreCon23} are dimension independent and are applicable to other matrix vector norms (all Schatten norms, in particular), but as far as we can tell they also make use of the higher regularity afforded by working with functions with bounded Hessian.
\end{rem}

\begin{rem}\label{rem:noBDhedgehog}
It would also be of interest to find out whether $u_H$ is extremal for $\TD_F$. However, the symmetrized gradient case brings several additional difficulties. Arguing as in Proposition \ref{prop:hedgehogTV}, one would be able to show that for $v,w$ with $u_H = \lambda v + (1-\lambda)w$ and $\TD_F(v)=\TD_F(w)=\TD_F(u_H)$, that $\E v, \E w$ should be absolutely continuous with respect to $\L^2$, that is $\E v = E v \L^2$, $\E w = E w \L^2$, and also 
\[Ev = \big( e_\varphi^\top Ev \, e_\varphi \big) (e_\varphi \otimes e_\varphi), \quad Ew = \big( e_\varphi^\top Ew\, e_\varphi \big) (e_\varphi \otimes e_\varphi).\]
However, from this point one cannot directly follow the same type of arguments to draw conclusions about $v,w$. First, in this case the complete derivatives $Dv,Dw$ exist only as distributions, which complicates drawing conclusions from them. Moreover, even sidestepping that difficulty, we would not be able to derive the relation $(v \cdot e_\varphi)/r = \nabla(v \cdot e_r)e_\varphi$ used above, since the symmetrization adds an extra term $\nabla(v \cdot e_\varphi)e_r$. Lastly, the example $z \in \mathcal{A}$ defined by $z(x_1,x_2) = (-x_2,x_1)$ for which $\nabla(z \cdot e_r)e_\varphi = \nabla(z \cdot e_\varphi)e_\varphi = 0$ shows that, even among spherically symmetric vector fields (that is, those for which $v(x)=R^\top v(Rx)$ for all $x \in B(0,1)$ and $R \in \SO(2)$), one cannot conclude $\nabla v = \E v$ from considerations based only on $\TD_F(v)$, which would allow to conclude as in the $\TV$ case without issues arising from cancellations. Let us point out however, that if one could prove the extremality of $u_H = e_r$ for $\TD_F$, then $e_\varphi$ would also be extremal, even in the restriction of the ball to divergence-free displacements
\begin{equation}\label{eq:divfreeball}\Big\{ [u] \in \BD(\Omega)/\mathcal{A} \,|\, \div u = 0 \text{ and }\TD(u) \ls \TD(e_\varphi)\Big\}.\end{equation}
To see this, it is enough to compose $u_H$ with a constant rotation of angle $\pi/2$ (which is a linear isomorphism), and notice that for any convex subset $C$ and any linear subspace $L$ within the same vector space we have
\[\Ext(C) \cap L \subset \Ext(C \cap L),\]
so adding the divergence-free constraint would preserve the extremality.
\end{rem}

\section*{Acknowledgements}
The Institute of Mathematics and Scientific Computing of the University of Graz, with which the first-named author is affiliated, is a member of NAWI Graz (\texttt{https://www.nawigraz.at/en/}).

\bibliographystyle{plain}
\bibliography{vectortv}

\end{document}